\documentclass[microtype]{gtpart}

\usepackage{pinlabel}
\usepackage[utf8]{inputenc}

\title[On a Spectral Sequence for the Cohomology of Infinite Loop
Spaces]{On a Spectral Sequence for the Cohomology of Infinite Loop
  Spaces}

\author[R Haugseng]{Rune Haugseng}
\givenname{Rune}
\surname{Haugseng}
\address{Department of Mathematical Sciences\\
University of Copenhagen\\
Universitetsparken 5\\
2100 København Ø\\
Denmark}
\email{haugseng@math.ku.dk}
\urladdr{http://sites.google.com/site/runehaugseng}

\author[H Miller]{Haynes Miller}
\givenname{Haynes}
\surname{Miller}
\address{Department of Mathematics\\
Massachusetts Institute of Technology\\
77 Massachusetts Avenue\\
Cambridge, MA 02139\\
USA}
\email{hrm@math.mit.edu}
\urladdr{http://math.mit.edu/~hrm}

 	
\arxivreference{1302.1816}
\arxivpassword{xixaa}

\volumenumber{}
\issuenumber{}
\publicationyear{}
\papernumber{}
\startpage{}
\endpage{}
\doi{}
\MR{}
\Zbl{}
\received{}
\revised{}
\accepted{}
\published{}
\publishedonline{}
\proposed{}
\seconded{}
\corresponding{}
\editor{}
\version{}

\usepackage{amsmath,amssymb,amsthm}

\usepackage{url}
\usepackage[shortlabels]{enumitem}
  \setitemize[1]{leftmargin=2em}
  \setenumerate[1]{leftmargin=*}
\usepackage{eucal}
\newtheorem{thm}{Theorem}[section]
\newtheorem{lemma}[thm]{Lemma}
\newtheorem{propn}[thm]{Proposition}
\newtheorem{cor}[thm]{Corollary}

\theoremstyle{definition}
\newtheorem{defn}[thm]{Definition}
\newtheorem{ex}[thm]{Example}

\newtheorem{remark}[thm]{Remark}

\newcommand{\blank}{\text{--}}

\usepackage{tikz-cd}

\newcommand{\defterm}[1]{\emph{#1}}
\newcommand{\isoto}{\xrightarrow{\sim}}

\newcommand{\IFF}{if and only if}

\newcommand{\catname}[1]{\ensuremath{\text{\textup{#1}}}}
\newcommand{\txt}[1]{\ensuremath{\text{\textup{#1}}}}

\newcommand{\Mod}{\catname{Mod}}

\newcommand{\Fun}{\txt{Fun}}

\newcommand{\xto}[1]{\xrightarrow{#1}}

\newcommand{\from}{\leftarrow}

\usepackage{tikz}
\usetikzlibrary{matrix,arrows}

\newcommand{\ltikzcd}[1]{{\setlength\mathsurround{0pt} \begin{tikzcd}#1\end{tikzcd}}}

\newcommand{\csquare}[8]{ %
\[ %
\begin{tikzpicture} %
\matrix (m) [matrix of math nodes,row sep=3em,column sep=2.5em,text height=1.5ex,text depth=0.25ex] %
{ #1 \pgfmatrixnextcell #2 \\ %
  #3 \pgfmatrixnextcell #4 \\ }; %
\path[->,font=\footnotesize] %
(m-1-1) edge node[auto] {$#5$} (m-1-2)%
(m-1-1) edge node[left] {$#6$} (m-2-1)%
(m-1-2) edge node[auto] {$#7$} (m-2-2)%
(m-2-1) edge node[below] {$#8$} (m-2-2);%
\end{tikzpicture}%
\]%
}

\newcommand{\nodispcsquare}[8]{ %
\begin{tikzpicture} %
\matrix (m) [matrix of math nodes,row sep=3em,column sep=2.5em,text height=1.5ex,text depth=0.25ex] %
{ #1 \pgfmatrixnextcell #2 \\ %
  #3 \pgfmatrixnextcell #4 \\ }; %
\path[->,font=\footnotesize] %
(m-1-1) edge node[auto] {$#5$} (m-1-2)%
(m-1-1) edge node[left] {$#6$} (m-2-1)%
(m-1-2) edge node[auto] {$#7$} (m-2-2)%
(m-2-1) edge node[below] {$#8$} (m-2-2);%
\end{tikzpicture}%
}

\newcommand{\nolabelcsquare}[4]{\csquare{#1}{#2}{#3}{#4}{}{}{}{}}

\newcommand{\ctriangle}[6]{ %
\[ %
\begin{tikzpicture} %
\matrix (m) [matrix of math nodes,row sep=3em,column sep=1.2em,text height=1.5ex,text depth=0.25ex] %
{  \pgfmatrixnextcell #1 \pgfmatrixnextcell \\ %
  #2 \pgfmatrixnextcell \pgfmatrixnextcell #3 \\ }; %
\path[->,font=\footnotesize] %
(m-1-2) edge node[left] {$#4$} (m-2-1)%
(m-1-2) edge node[right] {$#5$} (m-2-3)%
(m-2-1) edge node[below] {$#6$} (m-2-3);%
\end{tikzpicture}%
\]%
}

\newcommand{\id}{\txt{id}}
\DeclareMathOperator{\colimP}{colim}
\newcommand{\colim}{\mathop{\colimP}}

\newcommand{\op}{\txt{op}}

\newcommand{\F}{\mathbb{F}_2}

\newcommand{\grV}{\txt{gr}\catname{Vect}}
\newcommand{\Vect}{\catname{Vect}}

\newcommand{\Restr}{\catname{Restr}}

\newcommand{\otimesL}{\otimes^{\mathbb{L}}}

\newcommand{\simp}{\boldsymbol{\Delta}}
\newcommand{\Sq}{\txt{Sq}}

\newcommand{\HH}{\mathrm{H}}
\newcommand{\Alg}{\catname{Alg}}
\newcommand{\ZZ}{\mathbb{Z}}
\begin{document}

\begin{abstract}
  We study the mod-2 cohomology spectral sequence arising from
  delooping the Bousfield-Kan cosimplicial space giving the
  2-nilpotent completion of a connective spectrum $X$.  Under good
  conditions its $E_{2}$-term is computable as certain non-abelian
  derived functors evaluated at $\HH^*(X)$ as a module over the
  Steenrod algebra, and it converges to the cohomology of
  $\Omega^\infty X$. We provide general methods for computing the
  $E_{2}$-term, including the construction of a multiplicative
  spectral sequence of Serre type for cofibration sequences of
  simplicial commutative algebras. Some simple examples are also
  considered; in particular, we show that the spectral sequence
  collapses at $E_{2}$ when $X$ is a suspension spectrum.
\end{abstract}

\begin{asciiabstract}
  We study the mod-2 cohomology spectral sequence arising from
  delooping the Bousfield-Kan cosimplicial space giving the
  2-nilpotent completion of a connective spectrum X.  Under good
  conditions its E_2-term is computable as certain non-abelian
  derived functors evaluated at H^*(X) as a module over the
  Steenrod algebra, and it converges to the cohomology of
  Omega^\infty X. We provide general methods for computing the
  E_2-term, including the construction of a multiplicative
  spectral sequence of Serre type for cofibration sequences of
  simplicial commutative algebras. Some simple examples are also
  considered; in particular, we show that the spectral sequence
  collapses at E_2 when X is a suspension spectrum.
\end{asciiabstract}

\maketitle


\section{Introduction}
This paper explores the relationship between the $\F$-cohomology
$\HH^{*}E = \HH^{*}(E; \F)$ of a connective spectrum $E$ and that of
its associated infinite loop space $\Omega^{\infty}E$.

The starting point is the stabilization map $\HH^*(E) \to
\HH^*(\Omega^{\infty}E)$, induced by the adjunction counit
$\Sigma^{\infty}\Omega^{\infty}E \to E$. This factors
through the maximal unstable quotient $D\HH^{*}(E)$ of the $A$-module
$\HH^{*}(E)$ (where $A$ is the Steenrod algebra), and this map then
extends over the free unstable algebra $UD\HH^*(E)$. This construction
provides the best approximation to $\HH^{*}(\Omega^{\infty}E)$ functorial
in the $A$-module $\HH^{*}(E)$. 

We study a spectral sequence that
converges (for $E$ connected and of finite type) to
$\HH^*(\Omega^{\infty}E)$ and has $E_{2}$-term given by the non-abelian
derived functors of $UD$ applied to $\HH^{*}(X)$. This is the cohomology
spectral sequence associated to the cosimplicial space obtained by
applying $\Omega^{\infty}$ to a cosimplicial Adams (or Bousfield-Kan)
resolution of the spectrum $E$.

This construction is analogous and in a sense dual to that
of \cite{MillerDeloop}, where the second author constructed a spectral
sequence that converges to $\HH_{*}(E)$ by forming a simplicial
resolution of $E$ by suspension spectra and applying the zero-space
functor $\Omega^\infty$. The best approximation to the homology of $E$
functorial in the homology of the infinite loop space $\Omega^{\infty}
E$ is given by the indecomposables of $\HH_{*}(\Omega^\infty E)$ with
respect to the Dyer-Lashof operations and products, which are
annihilated by the natural map $\HH_{*}(\Omega^{\infty}E) \to
\HH_{*}(E)$, and the $E^{2}$-term of the spectral sequence is given by
the non-abelian left derived functors of these indecomposables applied
to $\HH_{*}(\Omega^{\infty}E)$.

The spectral sequence we study here is hardly new, and has been
previously considered (in unpublished work) by Bill Dwyer, Paul
Goerss, and no doubt others. Our main contribution here is related to
the computation of the $E_{2}$-term, which is of the form
$\pi_{*}(UV_{\bullet})$, where $V_{\bullet}$ is a simplicial unstable
$A$-module. We show that this is determined by a natural
short exact sequence of graded unstable modules over the Steenrod
algebra, in which the end terms are explicitly given in terms of the
graded $A$-module $\pi_{*}(V_{\bullet})$. This yields an explicit but
mildly non-functorial description of the $E_{2}$-term.

This reduces the analysis of the $E_{2}$-term of the spectral
sequence to the computation of the derived functors $\mathbb{L}_{*}D$.
These derived functors of destabilization have been studied by many authors,
including Singer~\cite{SingerLoops2,SingerNewChCx}, Lannes and
Zarati~\cite{LannesZaratiDestab}, Goerss~\cite{GoerssUnstProj}, Kuhn
and McCarty~\cite{KuhnMcCarty}, and Powell~\cite{Powell}.

As an outcome of our computation, we find that the spectral sequence
must collapse when $X$ is a connected suspension spectrum,
$X=\Sigma^\infty B$ for $B$ a connected space.  While the spectral
sequence collapses by construction when $X$ is a mod-2
Eilenberg-Mac Lane spectrum, its collapse for suspension spectra is a
bit of a surprise.  This does not yet constitute an independent
calculation of the cohomology of $\Omega^\infty\Sigma^\infty B$,
however, since to prove that the spectral sequence collapses we simply
compare the size of the $E_2$ term with that of the known homology of
$\Omega^\infty\Sigma^\infty B$. It is possible that the collapse
follows from Dwyer's description \cite{DwyerDivSqSpSeq} of the
behavior of differentials in a spectral sequence of this
type.

It would be interesting to compare the spectral sequence we study to
that arising from the Goodwillie-Taylor tower of the functor
$\Sigma^\infty\Omega^\infty$, as studied by Kuhn and McCarty
\cite{KuhnMcCarty}. Those authors also relate their spectral sequence
to derived functors of destabilization, though in a less direct way
than they occur in our spectral sequence; we would like to better
understand the relationship between these constructions, which seems
analogous to the relationship between the Bousfield-Kan unstable Adams
spectral sequence and the spectral sequence arising from the lower
central series.

\subsection{Overview}
We construct the spectral sequence in \S\ref{sec:spseq}, and review
some background material on simplicial commutative $\F$-algebras in
\S\ref{sec:sgcas}. Then in \S\ref{sec:derU} we compute $\pi_{*}U(M)$
in terms of $\pi_{*}M$, where $M$ is any simplicial $A$-module. We end
by discussing some simple examples of the spectral sequence in
\S\ref{sec:Ex}.

\subsection{Acknowledgements}\ \newline \emph{RGH:} I thank Vigleik
Angeltveit, Paul Goerss, and Justin Noel for helpful conversations
about this project. I also thank the American-Scandinavian Foundation
and the Norway-America Association for partially supporting
me during the time much of this work was carried out. \\
\emph{HRM:} I am grateful to the members of the Centre for Symmetry
and Deformation at the University of Copenhagen for their hospitality
in May 2011, when the initial stages of this work were carried
out. The project was jump-started by conversations with Nick Kuhn at a
BIRS workshop, and benefited from guidance from Bill Dwyer along the
way. This research was carried out in part under NSF grant 0905950.

\section{Definition and Convergence of the Spectral
  Sequence}\label{sec:spseq}

In this section we define the spectral sequence we are interested in,
observe that its $E_{2}$-term is described by certain derived
functors, and show that it converges under suitable finiteness and
connectivity assumptions. More precisely, our goal
is to prove the following:
\begin{thm}\label{thm:Conv}
  Suppose $X$ is a connected spectrum of finite type, i.e. $\pi_{*}X$ is
  $0$ for $* \leq 0$ and is a finitely generated abelian group for $*
  > 0$. Then there is a
  convergent spectral sequence
  \[ E^{s,t}_{2} = \mathbb{L}_{-s}(UD)(\HH^{*}X)^{t} \Rightarrow
  \HH^{t+s}(\Omega^{\infty}X).\]
\end{thm}
Here $\mathbb{L}_{*}(UD)$ denotes the non-abelian derived functors of
$UD$, which can be defined as $\pi_{*}UD(M_{\bullet})$ where
$M_{\bullet}$ is the simplicial free resolution of the $A$-modules
$\HH^{*}X$.

To define the spectral sequence, recall that for any spectrum $X$ the
Eilenberg-Mac Lane ring spectrum $\HH\F$ gives a cosimplicial spectrum
\[P^{\bullet} := X \wedge \HH\F^{\wedge (\bullet+1)}.\] The homotopy
limit of $P^{\bullet}$ is the \emph{2-nilpotent completion}
$X^{\wedge}_{2}$ of $X$. Since the functor $\Omega^{\infty}$ preserves
homotopy limits, the cosimplicial space $\Omega^{\infty}P^{\bullet}$
has homotopy limit $\Omega^{\infty}(X^{\wedge}_{2})$. This gives a
spectral sequence in cohomology,
\[ E^{s,t}_{2} = \pi_{-s}\HH^{t}(\Omega^{\infty}P^{\bullet})
\Rightarrow \HH^{t+s}(\Omega^{\infty}(X^{\wedge}_{2})).\]

\begin{propn}[Bousfield]
  If $X$ is a connected spectrum of finite type (i.e. all its homotopy
  groups are finitely generated), then this spectral sequence
  converges.
\end{propn}
\begin{proof}
  This follows from (the dual of) the convergence result of
  \cite[\S4.5]{BousfieldHlgySpSeq}.
\end{proof}

\begin{lemma}
  Suppose $X$ is a connected spectrum of finite type. Then the map
  $\Omega^{\infty}X \to \Omega^{\infty}(X^{\wedge}_{2})$ exhibits
  $\Omega^{\infty}(X^{\wedge}_{2})$ as the $\HH\F$-localization of
  $\Omega^{\infty}X$. In particular, it induces an equivalence in
  $\HH\F$-cohomology.
\end{lemma}
\begin{proof}
  If $X$ is connected, then by \cite[Theorem 6.6]{BousfieldLocSpa}
  the 2-nilpotent completion $X^{\wedge}_{2}$ is equivalent to the
  $\HH\F$-localization of $X$; in particular the natural map $X \to
  X^{\wedge}_{2}$ induces an equivalence in $\HH\F$-cohomology.

  Moreover, under the stated assumptions on $X$ the map $X \to
  X^{\wedge}_{2}$ induces an isomorphism $(\pi_{*}X) \otimes
  \mathbb{Z}^{\wedge}_{2} \isoto \pi_{*}X^{\wedge}_{2}$, by
  \cite[Proposition 2.5]{BousfieldLocSpa}. Since $X$ is connected, the
  space $\Omega^{\infty}X$ is nilpotent, and so by \cite[Example
  VI.5.2]{BousfieldKanCompl} the map $\Omega^{\infty}X \to
  (\Omega^{\infty}X)^{\wedge}_{2}$ also induces an isomorphism
  $(\pi_{*}\Omega^{\infty}X) \otimes \mathbb{Z}^{\wedge}_{2} \isoto
  \pi_{*}(\Omega^{\infty}X)^{\wedge}_{2}$. Since
  $\Omega^{\infty}(X^{\wedge}_{2})$ is 2-complete, the map
  $\Omega^{\infty}X \to \Omega^{\infty}(X^{\wedge}_{2})$ factors
  through $(\Omega^{\infty}X)^{\wedge}_{2}$; in the resulting
  commutative diagram \ctriangle{\pi_{*}(\Omega^{\infty}X) \otimes
    \mathbb{Z}^{\wedge}_{2}}{\pi_{*}(\Omega^{\infty}X)^{\wedge}_{2}}{\pi_{*}(\Omega^{\infty}(X^{\wedge}_{2}))}{}{}{}
  we know that two of the maps are isomorphisms, hence the map
  $(\Omega^{\infty}X)^{\wedge}_{2} \to
  \Omega^{\infty}(X^{\wedge}_{2})$ is a weak equivalence. The result
  follows since under our assumptions the map $\Omega^{\infty}X \to
  (\Omega^{\infty}X)^{\wedge}_{2}$ exhibits
  $(\Omega^{\infty}X)^{\wedge}_{2}$ as the $\HH\F$-localization of
  $\Omega^{\infty}X$ by \cite[Proposition VI.5.3]{BousfieldKanCompl}.
\end{proof}
Under these finiteness assumptions the spectral sequence
thus converges to the mod-2 cohomology of $\Omega^{\infty}X$. To
describe the $E_{2}$-term more algebraically, we appeal to Serre's
computation of the cohomology of Eilenberg-Mac Lane spaces. To state
this we must first recall some definitions:
\begin{defn}
  Let $\Mod_{A}$ be the category of (graded) $A$-modules, and let
  $\mathcal{U}$ be the full subcategory of \emph{unstable} modules,
  i.e. $A$-modules $M$ such that if $x \in M_{n}$ then $\Sq^{i}x = 0$
  for $i > n$. We define $D \colon \Mod_{A} \to \mathcal{U}$ to be the \emph{destabilization}
  functor, which sends an $A$-module $M$ to its quotient by the submodule generated
  by $\Sq^{i}x$ where $x \in M_{n}$ and $i > n$; the functor $D$ is left adjoint to the
  inclusion $\mathcal{U} \hookrightarrow \Mod_{A}$.
\end{defn}

\begin{defn}
  Let $\mathcal{K}$ be the category of unstable algebras over the
  Steenrod algebra $A$, i.e. augmented commutative $A$-algebras $R$
  that are unstable as $A$-modules, with $x^{2} = \Sq^{n}x$ for all $x
  \in R_{n}$. We define $U \colon \mathcal{U} \to \mathcal{K}$ to be
  the \emph{free unstable algebra} functor, which sends $M \in
  \mathcal{U}$ to \[S(M) / (x^{2} - \Sq^{|x|}x),\] where $S$ is the
  free graded symmetric algebra functor; this functor is left adjoint
  to the forgetful functor $\mathcal{K} \to \mathcal{U}$.
\end{defn}
 
\begin{thm}[Serre~\cite{SerreCohlgyEM}]\label{thm:Serre}
  If $M$ is an Eilenberg-Mac Lane spectrum of finite type, then the
  natural map $\HH^{*}(M) \to \HH^{*}(\Omega^{\infty}M)$ induces an
  isomorphism $UD(\HH^{*}M) \isoto \HH^{*}(\Omega^{\infty}M)$.
\end{thm}

For any $n > 0$ the spectrum $X \wedge \HH\F^{\wedge n}$ is a wedge of
suspensions of Eilenberg-Mac Lane spectra, so this theorem allows us
to rewrite the $E_{2}$-term of our spectral sequence as
\[E^{s,t}_{2} = \pi_{-s}UD(\HH^{*}P^{\bullet})^{t}.\] But by the
Künneth theorem $\HH^{*}(P^{n})$ is isomorphic to $\HH^{*}(X) \otimes A^{\otimes
  n+1}$, and in fact the simplicial $A$-module $\HH^{*}(P^{\bullet})$ is 
the standard cotriple
resolution of $\HH^{*}X$. The $A$-modules
$\pi_{*}UD(\HH^{*}P^{\bullet})$ can therefore be interpreted as the
(nonabelian) derived functors $\mathbb{L}_{*}(UD)$ of $UD$ evaluated
at $\HH^{*}X$. This completes the proof of Theorem~\ref{thm:Conv}.

\begin{remark}
  Our spectral sequence is of the type considered by Dwyer in
  \cite{DwyerDivSqSpSeq}, so by \cite[Proposition
  2.3]{DwyerDivSqSpSeq} it is a spectral sequence of $A$-algebras.  By
  (the dual of) results of Hackney~\cite{HackneyCosInfLoop} it is
  actually a spectral sequence of Hopf algebras.
\end{remark}

\section{Simplicial Commutative $\F$-Algebras}\label{sec:sgcas}
In this section we first review some background material on simplicial
commutative (graded) algebras: we recall the
model category structure on simplicial commutative algebras in
\S\ref{subsec:ModCat} and
in \S\ref{subsec:ops} we review the higher divided square operations
in the homotopy groups of simplicial commutative algebras. Then in
\S\ref{subsec:filtabstr} we discuss filtered algebras and modules from
an abstract point of view, and finally in \S\ref{subsec:SerreSpSeq} we
use this material to construct a ``Serre spectral sequence'' for
cofibre sequences of simplicial commutative algebras.

\subsection{Model Category Structure}\label{subsec:ModCat}
We will make use of a model category structure on simplicial augmented
commutative graded $\F$-algebras. This is an instance of a general
class of model categories constructed by
Quillen~\cite{QuillenHtpclAlg}, and is also described by
Miller~\cite{MillerSullivan}:

\begin{thm}
  There is a simplicial model category structure on the category of
  simplicial augmented graded commutative $\F$-algebras where a
  morphism is a weak equivalence or fibration if the underlying map of
  simplicial sets is a weak equivalence or Kan fibration.
\end{thm}

\begin{remark}
  For us \emph{graded} will mean $\mathbb{N}$-graded rather than
  $\mathbb{Z}$-graded. To avoid confusion, let us also mention that we
  de not require that a graded $\F$-algebra $A$ has $A_{0} = \F$, as
  is sometimes assumed in the literature.
\end{remark}

\begin{remark}
  Since a simplicial graded commutative $\F$-algebra is a simplicial
  group, a morphism $f \colon A \to B$ is a fibration \IFF{} the
  induced map $A \to B \times_{\pi_{0}B} \pi_{0}A$ is surjective. In
  particular, every object is fibrant.
\end{remark}

\begin{thm}[Rezk]\label{propn:ProperModCat}
  This model structure on simplicial augmented graded commutative
  $\F$-algebras is proper.
\end{thm}
\begin{proof}
  This follows from the properness criterion of \cite[Theorem
  9.1]{RezkSimplAlg}, since polynomial algebras are flat and thus
  tensoring with them preserves weak equivalences.
\end{proof}

We now recall Miller's description of the cofibrations in this model
category:
\begin{defn}\label{defn:almost}
  A morphism $f \colon A \to B$ of simplicial augmented commutative
  $\F$-algebras is \defterm{almost-free} if for every $n \ge 0$ there
  is a subspace $V_{n}$ of the augmentation ideal $IB_{n}$ and
  maps \[\delta_{i} \colon V_{n} \to V_{n-1}, \quad 1 \leq i \leq
  n, \]
  \[\sigma_{i} \colon V_{n} \to V_{n+1}, \quad 0 \leq i \leq n,\]
  so that the induced map $A_{n} \otimes S(V_{n}) \to B_{n}$ is an
  isomorphism for all $n$ and the following diagrams commute:
  \[\nodispcsquare{A_n \otimes S(V_n)}{B_n}{A_{n-1} \otimes
    S(V_{n-1})}{B_{n-1},}{}{d_i \otimes S(\delta_{i})}{d_i}{} \qquad
  \nodispcsquare{A_n \otimes S(V_n)}{B_n}{A_{n+1} \otimes
    S(V_{n+1})}{B_{n+1}.}{}{s_i \otimes S(\sigma_{i})}{s_i}{}\]
  In other words, all the face and degeneracy maps \emph{except} $d_{0}$ are
  induced from maps between the $V_{n}$'s.
\end{defn}

\begin{remark}
  The definition of almost free morphisms in \cite{MillerSullivan} is
  wrong and was corrected in \cite{MillerSullivanCorr}.
\end{remark}

\begin{thm}[Miller, {\cite[Corollary 3.5]{MillerSullivan}}]
  A morphism of simplicial augmented commutative $\F$-algebras is a
  cofibration \IFF{} it is a retract of an almost-free morphism.
\end{thm}

\subsection{Higher Divided Square Operations}\label{subsec:ops}
In this subsection we review the \emph{higher divided square}
operations on the homotopy groups of simplicial commutative
$\F$-algebras. These operations were initially introduced by
Cartan~\cite{CartanPuissDiv}, and have subsequently also been studied
by Bousfield~\cite{BousfieldOpsDerFtr} and Dwyer~\cite{DwyerHtpyOps}.

\begin{defn}
  If $V$ is a simplicial $\F$-vector space, we write $C(V)$ for the
  unnormalized chain complex of $V$ (obtained by taking the
  alternating sum of the face maps as the differential) and $N(V)$ for
  the \emph{normalized} chain complex, given by \[N_{k}X = \bigcap_{i
    \neq 0} \ker(d_{i} \colon X_{k} \to X_{k-1})\] with differential
  $d_{0}$). These chain complexes are quasi-isomorphic, and their
  homology groups are the same as the homotopy groups of $V$, regarded
  as a simplicial set.
\end{defn}

\begin{thm}[Dwyer~\cite{DwyerHtpyOps}] \label{thm:deltaprops}
  Let $A$ be a simplicial commutative $\F$-algebra.
  \begin{enumerate}[(i)]
  \item There are maps $\delta_{i} \colon C(A)_{n} \to N(A)_{n+i}$, $i\geq1$,
    that satisfy
    \[d \delta_{i}(a) =
  \begin{cases}
    \delta_{i}(d a), & n > i > 1, \\
    \delta_{1}(d a) + \phi(a), & i = 1, n > 1, \\
    a\, d a, & i = n > 1,\\
    a\, d a + \phi(a), & n = i = 1.
  \end{cases}\]
  Here $\phi(a)$ denotes the image in $N(A)$ of the square $a^{2}$ of $a$ in the
  multiplication on $A_{n}$.
\item In particular, there are \emph{higher divided square} operations
  $\delta_{i} \colon \pi_{n}(A) \to \pi_{n+i}(A)$ for $2 \le i \le
  n$. If $a^{2} = 0$ for all $a \in A$, there is also an operation
  $\delta_{1} \colon \pi_{n}(A) \to \pi_{n+1}(A)$ for $n \ge 1$.
\item These
  operations have the following properties:
  \begin{enumerate}[(1)]
  \item $\delta_{i} \colon \pi_{n}A \to \pi_{n+i}A$ is an additive homomorphism
    for $2 \leq i < n$, and $\delta_{n}$ satisfies
    \[ \delta_{n}(x+y) = \delta_{n}(x) + \delta_{n}(y) + xy.\]
  \item $\delta_{i}$ acts on products as follows:
    \[ \delta_{i}(xy) =
    \begin{cases}
      x^{2}\delta_{i}(y), & x \in \pi_{0}A,\\
      y^{2}\delta_{i}(x), & y \in \pi_{0}A,\\
      0, & \text{otherwise}.
    \end{cases}
    \]
  \item (``Ad\'{e}m relations'') If $i < 2j$ then
    \[ \delta_{i}\delta_{j}(x) = \sum_{(i+1)/2 \leq s \leq (i+j)/3}
    \binom{j-i+s-1}{j-s} \delta_{i+j-s}\delta_{s}(x).\]
  \end{enumerate}
\end{enumerate}
\end{thm}

\begin{remark}
  Part (i) is not quite true using Dwyer's definition of the chain-level
  operations in
  \cite{DwyerHtpyOps}, but is correct for the variant due to Goerss
  \cite{GoerssAQ}.
\end{remark}

\begin{remark}
  The upper bound in the ``Adém relation'' above differs from that in
  \cite{DwyerHtpyOps}, which does not give a sum of admissible
  operations; this form of the relation was proved by Goerss and
  Lada~\cite{GoerssLada} and implies, by the same proof as for
  Steenrod operations, that composites of $\delta$-operations are
  spanned by admissible composites:
\end{remark}

\begin{defn}
  A sequence $I = (i_{1},\ldots, i_{k})$ is \defterm{admissible} if
  $i_{t} \geq 2 i_{t+1}$ for all $t$. A composite $\delta_{I} :=
  \delta_{i_{1}} \delta_{i_{2}} \cdots \delta_{i_{k}}$ is
  \defterm{admissible} if $I$ is.
\end{defn}

\begin{cor}
  Any composite of $\delta$-operations can be written as a sum of
  admissible ones.
\end{cor}

\begin{remark}
  For any $x$ in $\pi_{*}A$ in positive degree we have $x^{2} =
  0$. Thus (iii)(1) and (iii)(2) imply that the top operation $\delta_{n}$ on
  $\pi_{n}$ is a divided square, whence the name ``higher divided
  squares'' for the $\delta_{i}$-operations.
\end{remark}

Dwyer proves Theorem~\ref{thm:deltaprops} by computing the homotopy
groups in the universal case, namely the symmetric algebra $s(V)$ on a
simplicial vector space $V$. We will now recall the result of this
computation, as well as the analogous result for exterior algebras
(both of which are originally due to
Bousfield~\cite{BousfieldOpsDerFtr}). To state this we make use of the
following theorem of Dold:
\begin{thm}[Dold~{\cite[5.17]{DoldSymmProd}}]\label{thm:Dold}
  Let $\Vect$ be the category of $\F$-vector spaces, and $\grV$ that
  of graded $\F$-vector spaces. For any functor $F \colon
  \Vect^{\times n} \to \Vect$ there exists a functor $\mathfrak{F}
  \colon \grV^{\times n} \to \grV$ such that for $V_{1},\ldots,V_{n}$
  simplicial $\F$-vector spaces there is a natural isomorphism
  \[ \pi_{*} F(V_{1}, \ldots V_{n}) \cong \mathfrak{F}(\pi_{*}V_{1},
  \ldots, \pi_{*}V_{n}),\]
  where on the left-hand side we take the homotopy of the diagonal of
  the multisimplicial $\F$-vector space $F(V_{1}, \ldots V_{n})$.
\end{thm}

\begin{ex}
  The Eilenberg-Zilber theorem implies that if $F$ is the
  tensor product functor, then $\mathfrak{F}$ is the graded tensor
  product of graded vector spaces.
\end{ex}

In the symmetric algebra case, the functor $\mathfrak{s}$ such that
$\pi_{*}s(V) = \mathfrak{s}(\pi_{*}V)$ has the following description:
\begin{thm}[Bousfield~\cite{BousfieldOpsDerFtr}, Dwyer~\cite{DwyerHtpyOps}]\label{thm:Bousfield-Dwyer}
  The functor $\mathfrak{s}$ sends a graded vector space $V$ to that
  freely generated on $V$ by a commutative product and operations $\delta_i$
  satisfying the relations stated in Theorem~\ref{thm:deltaprops} 
  above as well as the relation $x^{2} = 0$ for all $x$ of positive degree.
\end{thm}

If $B$ is a graded basis for $V$, then $\mathfrak{s}(\pi_*(V))$ is the free
commutative algebra (modulo the relation $x^2=0$ for $|x|>0$) 
generated by elements $\delta_{I}v$ in degree 
$|v| + i_{1} + \cdots + i_{k}$ for admissible sequences 
$I = (i_{1},\ldots,i_{k})$ with $i_{k} \ge 2$   of 
\defterm{excess} $e(I) := i_{1} - i_{2} - \cdots - i_{k}$ at most
$|v|$, as $v$ runs over $B$.

Let $s_{k}(V)$ be the subspace of the symmetric algebra $s(V)$ spanned
by products of length $k$; it also is a functor $\Vect \to \Vect$.
Implicit in Theorem \ref{thm:Bousfield-Dwyer} is
the following description of the functor $\mathfrak{s}_{k}$ such
that $\pi_{*}s_{k}(V) = \mathfrak{s}_{k}(\pi_{*}V)$.
\begin{thm}
  Suppose $V$ is a graded vector space. Define inductively a weight
  function on products $\delta_{I_{1}}(v_{1})\cdots
  \delta_{I_{n}}(v_{n})$ where $v_{i} \in V$ and the $I_{i}$'s are
  admissible sequences by
  \[ \text{wt}(v) = 1 \text{ for $v$ in $V$}, \]
  \[ \text{wt}(xy) = \text{wt}(x) + \text{wt}(y), \]
  \[ \text{wt}(\delta_{i}(x)) = 2 \text{wt}(x).\]
  Then $\mathfrak{s}_{k}(V)$ is the subspace of $\mathfrak{s}(V)$
  spanned by elements of weight $k$.
\end{thm}

Let $e(V)$ denote the exterior algebra on $V$ and $e_{k}(V)$ its
subspace of products of length $k$. Then there are functors
$\mathfrak{e}$ and $\mathfrak{e}_{k}$ such that $\pi_{*}e(V) =
\mathfrak{e}(\pi_{*}V)$ and $\pi_{*}e_{k}(V) =
\mathfrak{e}_{k}(\pi_{*}V)$ for a simplicial vector space $V$. These
were also computed by Bousfield:
\begin{thm}[Bousfield~\cite{BousfieldOpsDerFtr}]
  The functor $\mathfrak{e}$ sends a graded vector space $V$ to that
  freely generated on $V$ by a commutative product and operations
  $\delta_{i}$ (now with $i=1$ allowed) satisfying the same relations
  as in the symmetric case, and with $x^{2} = 0$ for all $x$. Thus
  $\mathfrak{e}(V)$ is generated by $v \in V$ and symbols
  $\delta_{I}v$ for admissible sequences $I = (i_{1},\ldots,i_{k})$
  (now with $i_{k} \geq 1$) of excess $\leq |v|$; the element
  $\delta_{I}v$ is again in degree $|v| + i_{1} + \cdots +
  i_{k}$. Defining the weight of such a generator as before, the
  graded vector space $\mathfrak{e}_{k}V$ is the subspace of
  $\mathfrak{e}V$ spanned by elements of weight $k$.
\end{thm}

\begin{remark}
  The same results hold in the graded case. We will use capital letters for
the graded versions of the functors considered above: so $S(V)$ denotes the 
free graded symmetric algebra on the graded vector space $V$, etc.
The higher divided power operations $\delta_{i}$ double the internal degree.
\end{remark}

\subsection{An Abstract Approach to Filtered Algebras and
  Modules}\label{subsec:filtabstr}
Given filtered algebras $A$, $B$, and $C$, and maps $A \to B$ and $A
\to C$, we would like to construct a filtration on the relative tensor
product $B \otimes_{A} C$ whose associated graded is the relative
tensor product $E^{0}B \otimes_{E^{0}A} E^{0}C$ of graded algebras,
where $E^{0}A$ denotes the associated graded algebra of the filtered
algebra $A$. Our goal in this section is to show that this is
possible, provided we allow ourselves to take cofibrant replacements
of these algebras in a suitable model category. We will do this by
considering filtered objects, and in
particular filtered modules over a filtered algebra, from an abstract
perspective.

Let $\mathbf{N}$ denote the partially ordered set of natural numbers
$0,1,\ldots$, considered as a category. If $\mathbf{C}$ is a category,
we write $\txt{Seq}(\mathbf{C})$ for the category $\Fun(\mathbf{N},
\mathbf{C})$ of sequences of morphisms in $\mathbf{C}$. A filtered
object of $\mathbf{C}$, if $\mathbf{C}$ is for example the category of
chain complexes of abelian groups, can then be thought of as a certain
kind of object of $\txt{Seq}(\mathbf{C})$.

Addition of natural numbers is a symmetric monoidal structure on
$\mathbf{N}$, so if $\mathbf{C}$ is a category with finite colimits
and a symmetric monoidal structure that commutes with finite colimits
in each variable (for short, $\mathbf{C}$ is a symmetric monoidal
category \emph{compatible with finite colimits}) we can equip
$\txt{Seq}(\mathbf{C})$ with the Day convolution tensor product. This
has as unit the constant sequence
\[ I \to I \to \cdots \] with value the unit $I$ in $\mathbf{C}$, and
if $A$ and $B$ are sequences in
$\mathbf{C}$ their tensor product $A \otimes B$ is given by
\[ (A \otimes B)_{n} = \colim_{i+j \leq n} A_{i} \otimes B_{j}.\]
\begin{remark}
  A simple cofinality argument shows that this colimit is isomorphic
  to the iterated pushout
  \[ A_{n} \otimes B_{0} \amalg_{A_{n-1} \otimes B_{0}} A_{n-1}
  \otimes B_{1} \amalg_{A_{n-2} \otimes B_{1}} \cdots \amalg_{A_{0}
    \otimes B_{n-1}} A_{0} \otimes B_{n}, \] which we can also
  describe as the coequalizer of the two obvious maps
  \[ \coprod_{s+t = n-1} A_{s} \otimes B_{t} \rightrightarrows
  \coprod_{i+j=n} A_{i} \otimes B_{j}.\] In other words, $(A \otimes
  B)_{n}$ is the quotient of $\coprod_{i+j = n} A_{i} \otimes B_{j}$
  where we identify the images of $A_{s} \otimes B_{t}$ with $s+t =
  n-1$ in $A_{s+1} \otimes B_{t}$ and $A_{s} \otimes B_{t+1}$.
\end{remark}

If $A$ is an algebra object in $\mathbf{C}$, the Day convolution on
$\txt{Seq}(\mathbf{C})$ induces a relative tensor product on the
category $\Mod_{A}(\txt{Seq}(\mathbf{C}))$ of $A$-modules, given by
the (reflexive) coequalizer
\[ M \otimes A \otimes N \rightrightarrows M \otimes N \to M
\otimes_{A} N,\] where $M$ and $N$ are $A$-modules in
$\txt{Seq}(\mathbf{C})$.  If $\mathbf{C}$ is, for instance, chain
complexes, then a filtered algebra in $\mathbf{C}$ is in particular an
algebra object of $\txt{Seq}(\mathbf{C})$, and filtered $A$-modules
$M$ and $N$ are also modules for $A$ in $\txt{Seq}(\mathbf{C})$. The
tensor product of $A$-modules then yields an object $M \otimes_{A}N$
in $\txt{Seq}(\mathbf{C})$, but in general this need not be a filtered
object of $\mathbf{C}$ --- the maps in this sequence need no longer be
monomorphisms. However, we can use a model structure on $\mathbf{C}$
to deal with this: If $\mathbf{C}$ is a combinatorial model category,
we can equip $\txt{Seq}(\mathbf{C})$ with the projective model
structure. A cofibrant object in $\txt{Seq}(\mathbf{C})$ is then a
sequence
\[ A_{0} \to A_{1} \to \cdots \]
where the objects $A_{i}$ are all cofibrant, and the morphisms $A_{i}
\to A_{i+1}$ are all cofibrations. If cofibrations in
$\mathbf{C}$ are monomorphisms, as they are for chain complexes or
simplicial algebras, then a cofibrant object of
$\txt{Seq}(\mathbf{C})$ is thus in particular a filtered object.

The Day convolution tensor product interacts well with this model
structure:
\begin{propn}[Isaacson]
  Let $\mathbf{C}$ be a symmetric monoidal combinatorial model
  category that satisfies the monoid axiom. Then
  $\txt{Seq}(\mathbf{C})$ is also a symmetric monoidal combinatorial
  model category with respect to the Day convolution and satisfies the
  monoid axiom.
\end{propn}
\begin{proof}
  This is a special case of Proposition 8.4 in the arXiv version of
  \cite{Isaacson} (which unfortunately does not appear in the
  published version).
\end{proof}

We can now apply results of Schwede and Shipley to get the following:
\begin{cor}\label{cor:modalgmodstr}
  Let $\mathbf{C}$ be a symmetric monoidal combinatorial model
  category that satisfies the monoid axiom, and suppose $A$ is a
  commutative algebra object in $\txt{Seq}(\mathbf{C})$. Then:
  \begin{enumerate}[(i)]
  \item The category $\Alg(\txt{Seq}{\mathbf{C}})$ of associative algebra
    objects of $\txt{Seq}(\mathbf{C})$ is a combinatorial model
    category. The forgetful functor to
    $\txt{Seq}(\mathbf{C})$ creates weak equivalences and fibrations,
    and the free-forgetful
    adjunction
    \[ \txt{Seq}(\mathbf{C}) \rightleftarrows
    \txt{Alg}(\txt{Seq}(\mathbf{C})) \] is a Quillen adjunction.
  \item If the unit of $\mathbf{C}$ is cofibrant, then the forgetful
    functor $\Alg(\txt{Seq}(\mathbf{C}) \to \txt{Seq}(\mathbf{C})$
    preserves cofibrant object.
  \item The category $\txt{Mod}_{A}(\txt{Seq}(\mathbf{C}))$ of
    $A$-modules is a symmetric monoidal combinatorial model
    category satisfying the monoid axiom. The forgetful functor to
    $\txt{Seq}(\mathbf{C})$ creates weak equivalences and fibrations,
    and the free-forgetful
    adjunction
    \[ F_{A} : \txt{Seq}(\mathbf{C}) \rightleftarrows
    \txt{Mod}_{A}(\txt{Seq}(\mathbf{C})) : U_{A}\] is a Quillen
    adjunction.
  \item If the underlying object of $A$ is cofibrant in
    $\txt{Seq}(\mathbf{C})$ then the forgetful functor $U_{A}$ also
    preserves cofibrations.
  \end{enumerate}
\end{cor}
\begin{proof}
  (i), (ii) and (iii) follow from \cite[Theorem
  4.1]{SchwedeShipleyAlgMod}, and (iv) is an easy consequence of the
  construction of the model structure using
  \cite[Lemma 2.3]{SchwedeShipleyAlgMod}: If $I$ is a set of
  generating cofibrations in $\txt{Seq}(\mathbf{C})$, then $F_{A}(I)$
  is a set of generating cofibrations in $A$-modules. The triple
  $U_{A}F_{A}$ is $A \otimes \blank$, which is a left Quillen functor
  if $A$ is cofibrant in $\txt{Seq}(\mathbf{C})$. Thus $U_{A}$ takes
  the generating cofibrations to cofibrations. But $U_{A}$ also
  preserves colimits, so as any cofibration is a transfinite composite
  of pushouts of generating cofibrations this means it preserves all
  cofibrations.
\end{proof}

\begin{cor}\label{cor:tenscofibt}
  Let $\mathbf{C}$ be a symmetric monoidal combinatorial model
  category that satisfies the monoid axiom, and suppose $A$ is a
  commutative algebra object in $\txt{Seq}(\mathbf{C})$ whose
  underlying object in $\txt{Seq}(\mathbf{C})$ is cofibrant. If $M$
  and $N$ are cofibrant $A$-modules, then $M \otimes_{A} N$ is a
  cofibrant object of $\txt{Seq}(\mathbf{C})$.
\end{cor}
This is the result we need to make our spectral sequence: if $A$ is a
suitable filtered algebra in, say, chain complexes, and $M$ and $N$
are filtered $A$-modules, we can take cofibrant replacements for them
in the model structure on $\txt{Mod}_{A}(\txt{Seq}(\mathbf{C})$ to get
a cofibrant relative tensor product over $A$, which is in particular a
filtered object and so gives a spectral sequence.

Next we want to analyze the associated graded object of such a
relative tensor product, which will allow us to describe the
$E^{1}$-page of our spectral sequence:
\begin{defn}
  Let $\mathbf{C}$ be a category with finite colimits and a zero
  object $0$. Write $\txt{Gr}(\mathbf{C})$ for the product $\prod_{i =
    0}^{\infty} \mathbf{C}$ and $\txt{Triv} \colon
  \txt{Gr}(\mathbf{C}) \to \txt{Seq}(\mathbf{C})$ for the functor that
  sends $(X_{i})_{i \in \mathbb{N}}$ to the sequence
  \[ X_{0} \xto{0} X_{1} \xto{0} \cdots.\] This has a left adjoint
  $E^{0} \colon \txt{Seq}(\mathbf{C}) \to \txt{Gr}(\mathbf{C})$, the
  \emph{associated graded} functor. We have $(E^{0}A)_{0} = A_{0}$ and
  $(E^{0}A)_{n}$ for $n > 0$ is the quotient $A_{n}/A_{n-1}$, i.e. the
  pushout
  \nolabelcsquare{A_{n-1}}{A_{n}}{0}{A_{n}/A_{n-1}.}
\end{defn}

If $\mathbf{C}$ has a symmetric monoidal structure, then we can equip
$\txt{Gr}(\mathbf{C})$ with a graded tensor product (another Day
convolution), given by \[(X \otimes Y)_{n} = \coprod_{i+j=n} X_{i} \otimes
Y_{j}.\]
The unit is $(I, 0, 0, \ldots)$.
\begin{propn}\label{propn:E0mon}
  Let $\mathbf{C}$ be a symmetric monoidal category compatible with
  finite colimits that has a zero object. Then the functor $E^{0}
  \colon \txt{Seq}(\mathbf{C}) \to \txt{Gr}(\mathbf{C})$ is symmetric
  monoidal.
\end{propn}
\begin{proof}
  $E^{0}$ clearly preserves the unit, so it suffices to show that
  there is a natural isomorphism $E^{0}M \otimes E^{0}N \isoto E^{0}(M
  \otimes N)$.

  By definition, $E^{0}_{n}(M \otimes N)$ is the cofibre of $(M
  \otimes N)_{n-1} \to (M \otimes N)_{n}$. For $n \in \mathbf{N}$, let
  $(\mathbf{N} \times \mathbf{N})_{\leq n}$ denote the full
  subcategory of $\mathbf{N} \times \mathbf{N}$ spanned by the objects
  $(i,j)$ with $i+j \leq n$; if $M \boxtimes N$ denotes the composite functor
  \[\mathbf{N} \times \mathbf{N} \xto{M \times N} \mathbf{C} \times
  \mathbf{C} \xto{\otimes} \mathbf{C},\] then $(M \otimes N)_{n}$ is
  by definition given by the colimit of $M \boxtimes N$ restricted to
  $(\mathbf{N} \times \mathbf{N})_{\leq n}$. Let $\alpha$ denote the
  inclusion $(\mathbf{N} \times \mathbf{N})_{\leq (n-1)}
  \hookrightarrow (\mathbf{N} \times \mathbf{N})_{\leq n}$; then $(M
  \otimes N)_{n-1}$ is isomorphic to the colimit of the left Kan
  extension $\alpha_{!}(M \boxtimes N)|_{(\mathbf{N} \times
    \mathbf{N})_{\leq (n-1)}}$. Thinking of $0$ as the constant
  diagram of shape $(\mathbf{N} \times \mathbf{N})_{\leq n}$ with
  value $0$, we can write $E^{0}_{n}(M \otimes N)$ as the pushout of
  two maps between colimits of diagrams of the same shape. Moreover,
  these maps arise from natural transformations, so since colimits commute we
  can identify $E^{0}_{n}(M \otimes N)$ with the colimit of the
  functor $\beta \colon (\mathbf{N} \times \mathbf{N})_{\leq n} \to \mathbf{C}$
  that assigns to $(i,j)$ the cofibre of the map
  \[ \phi_{i,j} \colon (\alpha_{!}(M \boxtimes N)|_{(\mathbf{N} \times
    \mathbf{N})_{\leq (n-1)}})(i,j) \to (M \boxtimes N)(i,j).\] If
  $i+j < n$ then, since $(\mathbf{N} \times \mathbf{N})_{\leq (n-1)}$
  is a full subcategory of $(\mathbf{N} \times \mathbf{N})_{\leq n}$,
  the map $\phi_{i,j}$ is an isomorphism, so $\beta(i,j) \cong 0$. It
  follows that the colimit of $\beta$ is just the coproduct
  $\coprod_{i+j = n} \beta(i,j)$, and it remains to show that
  $\beta(i,j)$ is isomorphic to $E^{0}_{i}M \otimes E^{0}_{j}N$.

  If $i + j = n$, let $(\mathbf{N} \times \mathbf{N})_{<(i,j)}$ be the full
  subcategory of $\mathbf{N} \times \mathbf{N}$ spanned by the objects
  $(x,y)$ with $x \leq i$ and $y \leq j$, except for $(i,j)$. Then by
  definition 
  $\alpha_{!}((M \boxtimes N)|_{(\mathbf{N} \times
    \mathbf{N})_{\leq (n-1)}})(i,j)$ is the colimit of $M \boxtimes N$
  restricted to $(\mathbf{N} \times \mathbf{N})_{<(i,j)}$.

  Write $(\mathbf{N} \times \mathbf{N})_{<(i,j)}^{0}$ for the full
  subcategory 
  \[ (i-1,j) \from (i-1,j-1) \to (i,j-1)\] of $(\mathbf{N} \times
  \mathbf{N})_{<(i,j)}$. We claim the inclusion $(\mathbf{N} \times
  \mathbf{N})_{<(i,j)}^{0} \hookrightarrow (\mathbf{N} \times
  \mathbf{N})_{<(i,j)}$ is cofinal, and so gives an isomorphism of
  colimits. By \cite[Theorem IX.3.1]{MacLaneWorking}, to see this it suffices to show
  that the categories \[((\mathbf{N} \times
  \mathbf{N})_{<(i,j)}^{0})_{(x,y)/} = (\mathbf{N} \times
  \mathbf{N})_{<(i,j)}^{0} \times_{(\mathbf{N} \times
    \mathbf{N})_{<(i,j)}} ((\mathbf{N} \times
  \mathbf{N})_{<(i,j)})_{(x,y)/}\] are non-empty and connected. But this
  category is either all of $(\mathbf{N} \times
  \mathbf{N})_{<(i,j)}^{0}$ if $x \leq i-1$ and $y \leq j-1$, or the
  single object $(i,j-1)$ if $x = i$, or the single object $(i-1,j)$
  if $y = j$; these are certainly all non-empty and connect. We may thus
  identify $\alpha_{!}(M \boxtimes N)|_{(\mathbf{N} \times
    \mathbf{N})_{\leq (n-1)}})(i,j)$ with the pushout $M_{i} \otimes
  N_{j-1} \amalg_{M_{i-1} \otimes N_{j-1}} M_{i-1} \otimes N_{j}$ and
  $\beta(i,j)$ with the total cofibre of the square \nolabelcsquare{M_{i-1}
    \otimes N_{j-1}}{M_{i-1} \otimes N_{j}}{M_{i} \otimes
    N_{j-1}}{M_{i} \otimes N_{j}.}  The cofibres of the columns here
  are $E^{0}_{i}M \otimes N_{j-1}$ and $E^{0}_{i}M \otimes N_{j}$,
  since the tensor product preserves colimits in each variable, and so
  the total cofibre $\beta(i,j)$ is isomorphic to the cofibre of the
  map $E^{0}_{i}M \otimes N_{j-1} \to E^{0}_{i}M \otimes N_{j}$, which
  is $E^{0}_{i}M \otimes E^{0}_{j}N$, as required.
\end{proof}

\begin{cor}
  Suppose $A$ is a commutative algebra object of
  $\txt{Seq}(\mathbf{C})$, where $\mathbf{C}$ is as above. Then the
  adjunction $E^{0} \dashv \txt{Triv}$ induces an adjunction
  \[ E^{0} : \Mod_{A}(\txt{Seq}(\mathbf{C})) \rightleftarrows
  \Mod_{E^{0}A}(\txt{Gr}(\mathbf{C})) : \txt{Triv} \]
  such that $E^{0}$ is symmetric monoidal.
\end{cor}

This is immediate from Proposition~\ref{propn:E0mon} and the following
easy formal observation:
\begin{lemma}\label{lem:modadj}
  Let $\mathbf{C}$ and $\mathbf{D}$ be symmetric monoidal categories,
  and suppose 
  \[ F : \mathbf{C} \rightleftarrows \mathbf{D} : G \]
  is an adjunction such that $F$ is symmetric monoidal. If $A$ is a
  commutative algebra object of $\mathbf{C}$, this induces an
  adjunction
  \[ F_{A} : \Mod_{A}(\mathbf{C}) \rightleftarrows
  \Mod_{FA}(\mathbf{D}) : G_{A} \] such that $F_{A}$ is symmetric
  monoidal.
\end{lemma}

This allows us to identify the associated graded of a relative tensor
product:
\begin{cor}\label{cor:E0tens}
  Let $\mathbf{C}$ be a symmetric monoidal category compatible with
  finite colimits that has a zero object. Suppose $A$ is a
  commutative algebra object in $\txt{Seq}(\mathbf{C})$ and that $M$
  and $N$ are $A$-modules. Then there is a natural
  isomorphism \[E^{0}(M \otimes_{A} N) \cong E^{0}M \otimes_{E^{0}A}
  E^{0}N.\]
\end{cor}

Finally, we check that the colimit of a relative tensor product is the
expected one:
\begin{propn}
  Suppose $\mathbf{C}$ is a symmetric monoidal category compatible
  with small colimits. Then the colimit functor $\txt{Seq}(\mathbf{C}) \to
  \mathbf{C}$ is symmetric monoidal.
\end{propn}
\begin{proof}
  The unit for the tensor product on $\txt{Seq}(\mathbf{C})$ is the
  constant sequence with value $I$, the unit for the tensor product on
  $\mathbf{C}$. Thus $\colim$ preserves the unit. It remains to show
  that the natural map $\colim_{n} (A \otimes B)_{n} \to \colim_{n}
  A_{n} \otimes \colim_{n} B_{n}$ is an isomorphism. But the object
  $\colim_{n} (A \otimes B)_{n}$ is clearly the colimit over $(i,j)
  \in \mathbf{N} \times \mathbf{N}$ of $A_{i} \otimes B_{j}$. Since
  the tensor product on $\mathbf{C}$ preserves colimits in each
  variable, this colimit is indeed equivalent to $(\colim_{i \in
    \mathbf{N}} A_{i})\otimes (\colim_{i \in \mathbf{N}} B_{i})$.
\end{proof}

Applying Lemma~\ref{lem:modadj}, we get:
\begin{cor}
  Let $\mathbf{C}$ be as above. Suppose $A$ is a commutative
  algebra object in $\txt{Seq}(\mathbf{C})$ with colimit
  $\overline{A}$. Then the colimit-constant adjunction induces an
  adjunction
  \[ \colim : \Mod_{A}(\txt{Seq}(\mathbf{C})) \rightleftarrows
  \Mod_{\overline{A}}(\mathbf{C}) : \txt{const}\]
  where the left adjoint is symmetric monoidal.
\end{cor}

\begin{cor}\label{cor:filttenscolim}
  Suppose $A$ is a commutative algebra object in
  $\txt{Seq}(\mathbf{C})$ with colimit $\overline{A}$, and $M$ and
  $N$ are $A$-modules with colimits $\overline{M}$ and
  $\overline{N}$. Then $\colim M \otimes_{A} N$ is
  naturally isomorphic to $\overline{M} \otimes_{\overline{A}} \overline{N}$.
\end{cor}

\subsection{A Serre Spectral Sequence for Simplicial Commutative
  Algebras}\label{subsec:SerreSpSeq}
In this subsection we construct a multiplicative ``Serre spectral
sequence'' for the homotopy groups of the cofibre of a cofibration of
simplicial commutative algebras. We derive this by studying a spectral
sequence for filtered modules over a filtered differential graded
algebra. Our spectral sequence has the same form as one constructed by
Quillen~\cite{QuillenHtpclAlg}, but his construction does not give the
multiplicative structure.

\begin{remark}
  We will implicitly assume that all filtrations we consider are
  non-negatively graded and \emph{exhaustive}, in the sense that if
  $F_{0}A \subseteq F_{1}A \subseteq \cdots$ is a filtration of $A$,
  then $A$ is the union of the subobjects $F_{i}A$.
\end{remark}

\begin{propn}\label{propn:filtalgE1}
  Suppose $A$ is a filtered commutative differential graded $k$-algebra,
  non-negatively graded, where $k$ is a field, and $B$ and $C$ are
  filtered $A$-modules, also non-negatively graded.
  \begin{enumerate}[(i)]
  \item If $B$ and $C$ are cofibrant in the model structure on
    $A$-modules in sequences of maps of chain complexes of
    Corollary~\ref{cor:modalgmodstr}, then the tensor
    product $B \otimes_{A} C$ has a canonical filtration with
    associated graded
    \[E^{0}_{*}(B \otimes_{A}C) \cong E^{0}_{*}B \otimes_{E^{0}_{*}A}
    E^{0}_{*}C.\]
  \item Suppose $B$ and $C$ are in addition filtered
    $A$-algebras. Then the filtration of (i) makes $B \otimes_{A} C$ a
    filtered algebra, so the associated spectral sequence is multiplicative.
  \end{enumerate}
\end{propn}
\begin{proof}
  Since $k$ is a field, every $k$-module is projective, hence in the
  projective model structure on the category $\txt{Ch}^{\geq 0}_{k}$
  of non-negatively graded chain complexes of $k$-modules every object
  is cofibrant. Thus in the projective model structure on
  $\txt{Seq}(\txt{Ch}^{\geq 0}_{k})$ the cofibrant objects are
  precisely those that are sequences of monomorphisms, i.e. those that
  correspond to filtered chain complexes. Part (i) then follows from
  Corollaries~\ref{cor:tenscofibt}, \ref{cor:E0tens} and
  \ref{cor:filttenscolim} applied to the projective model structure on
  chain complexes of $k$-modules.

  If $B$ and $C$ are filtered $A$-algebras, then we may regard them as
  associative algebra objects in the category
  $\Mod_{A}(\txt{Seq}(\txt{Ch}^{\geq 0}_{k}))$. Their relative tensor
  product is then also an associative algebra object in this category,
  and by (i) its underlying object of $\txt{Seq}(\txt{Ch}^{\geq
    0}_{k})$ corresponds to a filtered chain complex. Thus $B
  \otimes_{A} C$ is a filtered algebra, and so yields a multiplicative
  spectral sequence.
\end{proof}

\begin{propn}\label{propn:degfiltE2}
  Suppose $A$ is a commutative differential graded $k$-algebra and $B$
  and $C$ are $A$-modules, all non-negatively graded. Filter $A$ and
  $B$ by degree, and give $C$ the trivial filtration with $F_{p}C = C$
  for all $p \geq 0$. Let $B'$ and $C'$ be cofibrant replacements of
  $B$ and $C$ as $A$-modules in $\txt{Seq}(\txt{Ch}_{k}^{\geq
    0})$. Then in the spectral sequence associated to the induced
  filtration on $B' \otimes_{A} C'$ we have
  \begin{enumerate}[(i)]
  \item $E^{1}_{s,t} = (B' \otimes_{A} \pi_{t-s}C)_{s}$, where $A$ acts
    on $\pi_{*}C$ via the map $A \to \pi_{0}A$.
  \item $E^{2}_{s,t} = \pi_{s}(B' \otimes_{A} \pi_{t-s}C)$.
  \end{enumerate}
\end{propn}
\begin{proof}
  The graded tensor product $E^{0}_{*}B' \otimes_{E^{0}_{*}A} E^{0}_{*}C'$ has in
  degree $(s, t)$ the coequalizer of
  \[ \bigoplus_{\substack{i + j + k = s \\
      \rho + \sigma + \tau = t}} E^{0}_{i}B'_{\rho} \otimes E^{0}_{j}A_{\sigma} \otimes E^{0}_{k}C'_{\tau}
  \rightrightarrows  \bigoplus_{\substack{m + n = s \\
      \alpha + \beta = t}} E^{0}_{m}B'_{\alpha} \otimes E^{0}_{n}C'_{\beta}.\] 
  In our case $E^{0}_{l}B'_{\gamma}$ and $E^{0}_{l}A_{\gamma}$ are zero unless $l =
  \gamma$, and $E^{0}_{l}C'_{\gamma}$ is zero unless $l = 0$, so this
  is the coequalizer of
  \[ \bigoplus_{i + j = s} B'_{i} \otimes A_{j} \otimes C'_{t-s}
  \rightrightarrows  B'_{s} \otimes C'_{t-s}.\] 
  Now observe that the map $A_{j} \otimes C'_{t-s} \to C'_{t-s}$ is zero
  unless $j = 0$, since $A_{j}$ is in filtration $j$ and so the
  product must lie in $(E^{0}_{j}C')_{t-s} = 0$. Thus we can describe this
  coequalizer as killing all elements of the form $a \cdot b$ with $a
  \in A$ and $b \in B'$, giving $(B' \otimes_{A} \pi_{0}A)_{s}
  \otimes_{\pi_{0}A} C'_{t-s} \cong (B' \otimes_{\pi_{0}A}
  C'_{t-s})_{s}$.
  
  The differential in $B' \otimes_{A} C'$ satisfies the Leibniz rule, so
  if $b \otimes c$ is in filtration $s$ then $d(b \otimes c) = db
  \otimes c + b \otimes dc$. Here $db$ is in lower filtration than $b$,
  since it is in lower degree, hence $d_{0}$ comes from the
  differential in $C'$. Thus \[E^{1}_{s,t} \cong (B' \otimes_{\pi_{0}A}
  \pi_{t-s}C)_{s}.\]
  Similarly, the next differential $d_{1}$ comes from the differential
  in $B$, giving
  \[E^{2}_{s,t} \cong \pi_{s}(B' \otimes_{\pi_{0}A}
  \pi_{t-s}C).\qedhere\]
\end{proof}

The Dold-Kan correspondence extends to a Quillen equivalence of model categories
(cf. \cite[\S 4]{SchwedeShipleyMon}) between simplicial modules and
chain complexes, where the weak equivalences in
the two categories are the $\pi_{*}$-isomorphisms and the
quasi-isomorphisms, respectively. Using these model structures we can
define \emph{derived tensor products} as follows:
\begin{defn}
  If $A$ is a simplicial graded $\F$-algebra, $M$ is a right $A$-module,
  and $N$ is a left $A$-module, then the \defterm{derived tensor
    product} $M \otimesL_{A} N$ is the homotopy colimit of the
  simplicial diagram given by the bar construction, $M \otimes
  A^{\otimes \bullet} \otimes N$. Similarly, if $A$ is an algebra in
  chain complexes of graded $\F$-vector spaces, $M$ is a right
  $A$-module, and $N$ is a left $A$-module, we define a derived tensor
  product $M \otimesL_{A} N$ as the analogous homotopy colimit.
\end{defn}

\begin{remark}
  In the simplicial case, the homotopy colimit is given by the
  diagonal of the bar construction.
\end{remark}

\begin{remark}
  If $M$ is a cofibrant $A$-module, then for any $N$ the derived
  tensor product $M \otimesL_{A} N$ is equivalent to the ordinary
  tensor product $M \otimes_{A} N$.
\end{remark}

\begin{lemma}\label{lem:dertens}
  Let $A$ be a simplicial graded algebra, $X$ a right $A$-module, and
  $Y$ a  left $A$-module. There is a natural quasi-isomorphism \[N(X)
  \otimesL_{N(A)} N(Y) \to N(X \otimesL_{A} Y).\]
\end{lemma}
\begin{proof}
  There is a natural quasi-isomorphism $N(U) \otimes N(V) \to N(U
  \otimes V)$ for all simplicial abelian groups $U$ and $V$. Thus
  there is a natural transformation of simplicial diagrams $N(X
  \otimes A^{\otimes \bullet} \otimes Y) \to NX \otimes (NA)^{\otimes
    \bullet} \otimes NY$ that is a quasi-isomorphism levelwise. This
  implies that the induced map on homotopy colimits is also a
  quasi-isomorphism.
\end{proof}

\begin{cor}\label{cor:simpltensspseq}
  Suppose given simplicial augmented graded $\F$-algebras $A$, $B$,
  and $C$, and maps $A \to B$ and $A \to C$. Then there is a
  multiplicative spectral sequence
  \[ E^{2}_{s,t} = \pi_{s}(B \otimesL_{A} \pi_{t-s}(C))
  \Rightarrow \pi_{t}(B \otimesL_{A} C).\]
\end{cor}
\begin{proof}
  Let $P$ be a cofibrant replacement for $NB$ as an $NA$-module in $\txt{Seq}(\txt{Ch}^{\geq 0}_{k})$. Then
  by Proposition~\ref{propn:degfiltE2} we have a spectral sequence
  \[ E^{2}_{s,t} = \pi_{s}(P \otimes_{NA} \pi_{t-s}NC) \Rightarrow
  \pi_{t-s}(P \otimes_{NA} NC),\] which is multiplicative by
  Proposition~\ref{propn:filtalgE1}. Since taking the colimit of a
  sequnce is a left Quillen functor, $P$ is also a cofibrant
  replacement for $NB$ as an $NA$-module, so we can write this as
  \[ E^{2}_{s,t} = \pi_{s}(NB \otimesL_{NA} \pi_{t-s}NC) \Rightarrow
  \pi_{t-s}(NB \otimesL_{NA} NC).\] By Lemma~\ref{lem:dertens} we have
  natural quasi-isomorphisms $N B \otimesL_{NA} NC \to N(B
  \otimesL_{A} C)$ and, since $\pi_{t-s}NC \cong \pi_{t-s}C$ is
  concentrated in a single degree, $N B \otimesL_{NA} \pi_{t-s}NC \to
  N(B \otimesL_{A} \pi_{t-s}C)$. Thus we have a natural isomorphism
  \[ E^{2}_{s,t} \cong \pi_{s}(B \otimesL_{A} \pi_{t-s}C) \Rightarrow
  \pi_{t-s}(B \otimesL_{A} C).\qedhere\]
\end{proof}

As observed by Turner~\cite[Proof of Lemma 3.1]{Turner}, this spectral
sequence can be used to get a ``Serre spectral sequence'' for
cofibration sequences of simplicial commutative algebras:
\begin{cor}[``Serre Spectral Sequence'']\label{cor:SerreSpSeq}
  Suppose $f \colon A \to B$ is a cofibration of simplicial augmented
  graded commutative $\F$-algebras with cofibre $C$ and
  $\pi_{0}A = \F$. Then there is a multiplicative spectral sequence
  \[\pi_{s}(C) \otimes_{\F} \pi_{t-s}(A) \Rightarrow \pi_{t}(B). \]
\end{cor}
\begin{proof}
  By Corollary~\ref{cor:simpltensspseq} there is a multiplicative
  spectral sequence
  \[ E^{2}_{s,t} = \pi_{s}(B \otimesL_{A} \pi_{t-s}A) \Rightarrow
  \pi_{t}(B). \] By definition $C \cong B \otimes_{A} \F$, and
  so \[C \otimes_{\F} \pi_{t}A \cong (B \otimes_{A} \F) \otimes_{\F}
  \pi_{t}A \cong B \otimes_{A} \pi_{t}A,\] which is isomorphic to $B
  \otimesL_{A} \pi_{t}A$ since $A \to B$ is a cofibration and the
  model structure on simplicial commutative algebras is left proper by
  Proposition~\ref{propn:ProperModCat}. Since $\F$ is a field we have
  \[\pi_{s}(C \otimes_{\F} \pi_{t-s}A) \cong \pi_{s}C \otimes_{\F}
  \pi_{t-s}A,\] and so we can rewrite the $E^{2}$-term of the
  spectral sequence as $E^{2}_{s,t} \cong \pi_{s}(C) \otimes_{\F}
  \pi_{t-s}(A)$. 
\end{proof}

\section{Derived Functors of $U$}\label{sec:derU}
Our goal in this section is to compute $\pi_{*}U(M)$ where $M$ is a
simplicial unstable $A$-module. As a simplicial commutative algebra,
$U(M)$ depends only on the top non-zero Steenrod operations in $M$; in
\S\ref{subsec:restrvsp} we consider graded vector spaces equipped with
only these operations, which we call \emph{restricted vector spaces},
and observe that a simplicial restricted vector space decomposes up to
weak equivalence as a coproduct of simple pieces. In
\S\ref{subsec:ComppiU} we compute the derived functors of $U$ for
these simpler objects, which gives a description of $\pi_{*}U(M)$ as a
graded commutative algebra with higher divided square
operations. Using this we then give a more functorial description in
$\pi_{*}U(M)$ in \S\ref{subsec:functderU}, which in particular lets
us identify the action of the Steenrod operations.

\subsection{Restricted Vector Spaces}
\label{subsec:restrvsp}
In this subsection we define restricted vector spaces and make some
observations about their structure --- in particular, we show that a
chain complex of restricted vector spaces always decomposes up to
quasi-isomorphism as a direct sum of certain very simple complexes.

\begin{defn}
  A \defterm{restricted vector space} (over $\F$) is a non-negatively
  graded vector space $V$ equipped with linear maps $\phi_{i} \colon
  V^{i} \to V^{2i}$ for all $i$, called the \defterm{restriction maps}
  of $V$, such that $\phi_{0}\colon V^{0} \to V^{0}$ is the
  identity. A homomorphism of restricted vector spaces $f \colon V \to
  W$ is a homomorphism of graded vector spaces such that
  $\phi_{i}f^{i} = f^{2i}\phi_{i}$ for all $i$. We write $\Restr$ for
  the category of restricted vector spaces and restricted vector space
  homomorphisms. 
\end{defn}

\begin{defn}
  For $n \leq 0$, let $F(n)$ be the free restricted vector space with one generator
  $\iota_{n}$ in degree $n$. Thus $F(n)^{2^{r}n} = \F$
  with $\phi_{2^{r}n} = \id$ and $F(n)^{i} = 0$ otherwise --- in
  particular $F(0)$ is just $\F$ in degree $0$.

  For $k,n > 0$ let $T(n,k)$ be the nilpotent restricted
  vector space with one generator $\iota_{n,k}$ in degree $n$ subject
  to $\phi^{k}\iota_{n,k} = 0$; that is, $F(n)/\phi^{k}$. Thus
  $T(n,k)^{2^{r}n} = \F$ for $r = 0, \ldots, k$ with $\phi_{2^{r}n} = \id$ 
  for $r = 0,
  \ldots, k-1$, and $T(n,k)^{i}$ is $0$ otherwise.
\end{defn}

\begin{defn}
  Let $V$ be a restricted vector space. A \defterm{basis} $S$ of $V$
  consists of sets $S^{i}$ of elements of $V^{i}$ such that $S^{0}$ is
  a basis for $V^{0}$ and if $i = 2^{r}p$ with $p$ odd, then the set
  $(S^{2^{r}p} \cup \phi(S^{2^{r-1}p}) \cup \cdots \cup
  \phi^{r}(S^{p})) \setminus \{0\}$ is a basis for $V^{i}$.
\end{defn}

\begin{remark}
  It is clear that any restricted vector space has a basis, since we
  can inductively choose complements of $\phi(V_{i})$ in
  $V_{2i}$. Equivalently, any restricted vector space decomposes as a
  direct sum of $F(n)$'s and $T(n,k)$'s.
\end{remark}

\begin{defn}
  Let $C(q)$ be the non-negatively graded chain complex of restricted
  vector spaces
  \[ \cdots \to 0 \to 0 \to F(q),\]
  with $F(q)$ in degree $0$, and let $C(q, k)$ be the chain complex
  \[ \cdots \to 0 \to F(2^{k}q) \hookrightarrow F(q),\]
  with $F(q)$ in degree $0$ and $F(2^{k}q)$ in degree $1$.
\end{defn}

Given a chain complex $C$, write $C[n]$ for the suspended chain 
complex with $C[n]_i=C_{i-n}$.  Then clearly
\[ \HH_{*}(C(q)[n]) \cong
\begin{cases}
  F(q), & * = n, \\
  0, & * \neq n,
\end{cases}
\]
\[ \HH_{*}(C(q,k)[n]) \cong
\begin{cases}
  T(q,k), & * = n, \\
  0, & * \neq n.
\end{cases}
\]

\begin{propn}\label{propn:chaincxdecomp}
  Any chain complex of restricted vector spaces is quasi-isomorphic to
  a direct sum of $C(q)[n]$'s and $C(q, k)[n]$'s.
\end{propn}

\begin{proof}
  Let $(V_{*}, d)$ be a chain complex of restricted vector
  spaces. Pick a basis $S_{i}$ of $\HH_{i}(V_{*})$. For $v \in
  S_{i}^{q}$ define $W_{v}$ to be $F(q)[i]$ if $\phi^{r}v$ is
  never zero, and $C(q, k)[i]$ if $\phi^{k}v = 0$ but
  $\phi^{r}v \neq 0$ for $r < k$. Let $\hat{v}$ be a lift of $v$ to
  $V_{i}$; in the first case $\hat{v}$ defines a map $\psi_{v} \colon
  W_{v} \to V_{\bullet}$. In the second case, since $\phi^{k}(v) = 0$
  we can pick $\hat{w} \in V_{i+1}$ such that $d(\hat{w}) =
  \phi^{k}(\hat{v})$; then $\hat{v}$ and $\hat{w}$ define a map
  $\psi_{v}$ from $W_{v}$ to $V_{\bullet}$. Let $W := \bigoplus_{v \in
    S}W_{v}$ and let $\psi \colon W \to V$ be $\bigoplus_{v \in S}
  \psi_{v}$. Then $\psi$ is a quasi-isomorphism, since it is clear
  that on homology $\psi_{v}$ induces the inclusion in
  $\HH_{i}(V_{\bullet}, d)$ of the subspace generated by $v$.
\end{proof}

By the Dold-Kan correspondence the category $\txt{Ch}(\Restr)_{\geq
  0}$ of non-negatively graded chain complexes of restricted vector
spaces is equivalent to the category of simplicial restricted vector
spaces. Let's write $K[n, q]$ and $K[n, q, k]$ for the simplicial
objects corresponding to $F(q)[n]$ and $C(q,k)[n]$, respectively,
under this equivalence; then Proposition~\ref{propn:chaincxdecomp}
corresponds to:
\begin{cor}\label{cor:RestrDecomp}
  Any simplicial restricted vector space is weakly homotopy equivalent
  to a coproduct of $K[n,q]$'s and $K[n,q,k]$'s.
\end{cor}

\subsection{Computation of the Derived Functors of
  $U$}\label{subsec:ComppiU}

In this subsection we will prove the main technical result of this paper: we
compute the homotopy groups of the free unstable $A$-algebra on a
simplicial unstable $A$-module. As an algebra, $U(V)$ depends only on
$V$ as a restricted vector space: it is the ``enveloping algebra'' of $V$,
given by the free graded
commutative algebra on $V$ subject to the relation $x^{2} = \phi(x)$,
i.e. $S(V)/(x^{2} = \phi(x))$, where $S(V)$ is the graded symmetric
algebra on $V$. If $V$ is a simplicial restricted vector space we may
ask about $\pi_*(UV)$. It does not depend functorially on $\pi_*(V)$;
we do not have Dold's theorem \ref{thm:Dold} working for us. We will
describe  $\pi_*(UV)$ in terms of $\pi_*(V)$, but not functorially. 
Our description will use the functor $\mathfrak{S}$, given to us by
Dold's theorem, such that $\pi_*(SV)=\mathfrak{S}(\pi_*(V))$. 
It is described in detail above, in Theorem~\ref{thm:Bousfield-Dwyer}.
We will also use the ``loops'' functor $\Omega$ and its first derived functor
$\Omega_1$, defined by the exact sequence
  \[ 0 \to \Sigma \Omega_{1}V \to \Phi V \xto{\phi} V \to \Sigma
  \Omega V \to 0. \]
where $\Phi$ denotes the ``doubling'' functor, 
$(\Phi V)^{2q}_{n} = V^{q}_{n}$ and $(\Phi V)^{2q+1}_{n} = 0$.

Here is the result:

\begin{thm}\label{thm:piU}
  If $V$ is a simplicial restricted vector space, then there is a
  (non-canonical) isomorphism 
\[
\pi_{*}UV
  \cong U(\pi_{0}V)[0] \otimes \mathfrak{S}(\Sigma\Omega \pi_{*>0} V) \otimes
  \mathfrak{S}((\Sigma \Omega_{1} \pi_{*>0}V)[1]),
\] 
where $[1]$ denotes a shift by 1 in the simplicial degree. (By
\emph{non-canonical} we mean that the isomorphism depends on a choice
of basis of $\pi_{*}V$.)
\end{thm}

By Corollary~\ref{cor:RestrDecomp} we know that any simplicial
restricted vector space is weakly equivalent to a coproduct of
$K[n,q]$'s and $K[n,q,k]$'s. To show that this carries over to a
decomposition of $U(M)$ up to weak equivalence, we observe $U$
preserves weak equivalences and colimits:
\begin{propn}\label{propn:Ucolim}
  $U$, considered as a functor from restricted vector spaces to graded
  commutative $\F$-algebras, preserves colimits and weak equivalences.
\end{propn}
\begin{proof}
  We first show that $U$ preserves colimits. Let $U'$ denote $U$,
  regarded as a functor from restricted vector spaces to augmented
  graded commutative $\F$-algebras. The forgetful functor from
  augmented algebras to algebras preserves colimits, so it suffices to
  show that $U'$ preserves colimits. But this is clear since $U'$ has
  a right adjoint, namely the augmentation ideal functor for augmented
  graded commutative $\F$-algebras, regarded as a functor to
  restricted vector spaces with restriction maps given by squaring.

  To see that $U$ preserves weak equivalences, consider the
  word-length filtration on $U(V)$ for $V$ a simplicial restricted
  vector space. This gives rise to a spectral sequence of the form
  \[ \pi_{t}E_{s}(V) = \mathfrak{E}_{s}(\pi_{*}V)_{t}\Rightarrow
  \pi_{t} U(V),\] where $\mathfrak{E}$ is as in \S\ref{subsec:ops}
  Moreover, since the filtration is natural in $V$, so is the spectral
  sequence. Thus a weak equivalence $f \colon V \to W$ of simplicial restricted vector
  spaces induces a morphism of spectral sequences that gives an
  isomorphism on the $E^{1}$-page, since this only depends on the
  homotopy of the simplicial restricted vector space.  This implies
  that the map is an isomorphism of spectral sequences and hence, as
  these spectral sequences converge, it follows that $U(f)$ is a weak
  equivalence of simplicial graded commutative algebras.
\end{proof}
Combining this with Corollary~\ref{cor:RestrDecomp} we see that
$U(M)$, for any simplicial restricted vector space $M$, is weakly
equivalent to a tensor product of $U(K[n,q])$'s and
$U(K[n,q,k])$'s. It thus suffices to prove Theorem~\ref{thm:piU} in
these two cases. We begin with the easiest case, namely
$\pi_{*}U(K[n,q])$ for $q > 0$. For this we need to recall the
explicit form of the Dold-Kan construction:
\begin{defn}
  Let $\mathbf{C}$ be an abelian category. The \defterm{Dold-Kan
    construction} \[K \colon \txt{Ch}(\mathbf{C})_{\geq 0} \to
  \Fun(\simp^{\op}, \mathbf{C})\] sends a non-negatively graded chain
  complex $A$ to the simplicial object $K(A)$ defined as follows: We
  set \[K(A)_{n} = \bigoplus_{\alpha \colon [n] \twoheadrightarrow
    [k]} A_{k},\] where the coproduct is over surjective maps out of
  $[n]$. Then a map $K(A)_{n} \to K(A)_{m}$ is described by a
  ``matrix'' of maps $f_{\alpha,\beta}\colon A_{k} \to A_{l}$ from the
  component corresponding to $\alpha \colon [n] \twoheadrightarrow
  [k]$ to that corresponding to $\beta \colon [m] \twoheadrightarrow
  [l]$. To define the map $\phi^{*} \colon K(A)_{n} \to K(A)_{m}$
  corresponding to $\phi \colon [m] \to [n]$ in $\simp$ we take this
  to be given by
  \[ f_{\alpha,\beta} :=
  \begin{cases}
    \id, & l=k, \beta = \alpha \phi \\
    d, & l=k-1, d^{0}\beta = \alpha \phi\\
    0, & \text{otherwise}.
  \end{cases}
  \]
\end{defn}
The Dold-Kan correspondence (cf. Dold~\cite{DoldSymmProd},
Dold-Puppe~\cite{DoldPuppeNichtFtr}, Kan~\cite{KanFtrsCSS}) is then
that the functor $K$ is an equivalence of categories, with inverse the
normalized chain complex functor.

\begin{lemma}\label{lem:UKnq}
  Let $\F[n,q]$ denote the chain complex of graded vector spaces that
  is $0$ except in the degree $n$, where it is $\F[q]$ (the graded
  vector space with $\F$ in degree $q$ and $0$ elsewhere). Then for $q> 0$
  we have $UK[n,q] \cong S(K\F[n,q])$, where $K\F[n,q]$ is the
  Dold-Kan construction for graded vector spaces applied to $\F[n,q]$.
  In particular, we have an isomorphism
  \[\pi_{*}UK[n,q] \cong \mathfrak{S}\F[n,q]\]
  for all $n$ and $q > 0$.
\end{lemma}
\begin{proof}
  From the definition of the Dold-Kan functor $K$ we have
\[ K[n,q]_{i} =  \bigoplus_{[i] \twoheadrightarrow [n]}
F(q) \]
and so for $q > 0$ we have
\[ UK[n,q]_{i} \cong \bigotimes_{[i] \twoheadrightarrow [n]} UF(q) \cong
\bigotimes_{[i] \twoheadrightarrow [n]} S\F[q] \cong S(\bigoplus_{[i]
  \twoheadrightarrow [n]} \F[q]) \cong S(K\F[n,q])_{i}. \] Moreover, the
simplicial structure maps in $UK[n,q]$ and $SK\F[n,q]$ are also
clearly the same (on ``components'' they are either the identity
or zero), so these simplicial graded vector spaces are isomorphic.
\end{proof}

For the case $q = 0$ the algebra $U F(0)$ is not a symmetric algebra,
since we impose the relation $x^{2} = x$ on the generator $x$: it is a
{\em Boolean algebra}. 
 If $V$ is a vector space, we write $b(V)$ for the free
  Boolean algebra $s(V)/(x^{2} = x)$ on $V$ (where $s(V)$ is the
  ungraded symmetric algebra on $V$). By Theorem~\ref{thm:Dold}, there
  is a functor $\mathfrak{b} \colon \grV \to \grV$ such that
  $\pi_{*}b(V) = \mathfrak{b}(\pi_{*}V)$ for $V$ a simplicial vector
  space.

\begin{lemma}\label{lem:BoolHtpy}
  Suppose $V$ is a graded $\F$-vector space. Then
  \[ \mathfrak{b}(V)_{*} =
  \begin{cases}
    b(V_{0}), & * = 0, \\
    0, & \txt{otherwise.}
  \end{cases}\]
\end{lemma}
\begin{proof}
  Suppose $V$ is a simplicial vector space. Then
  Theorem~\ref{thm:deltaprops} implies that for any element $a$ in
  degree $n > 0$ in the chain complex associated to $b(V)$ such that $da
  =0$, there exists an element $\delta_{1}a$ in degree $n+1$ such that $d\delta_{1}a =
  \phi(a) = a$. Thus $\pi_{*}b(V) = 0$ for $* > 0$.
\end{proof}

\begin{lemma}\label{lem:piUKnzero}
  For any $n$ we have an isomorphism $UK[n,0] \cong b(K\F[n])[0]$, the 
  simplicial graded vector space with $b(K\F[n])$ in degree 0, and so \[\pi_{*}UK[n,0] =
  \begin{cases}
    (b(\F))[0], & n=0, \\
    b(0)[0] \cong \F[0], & n \neq 0.
  \end{cases}\]
\end{lemma}
\begin{proof}
  As in the proof of Lemma~\ref{lem:UKnq} we have
  \[ UK[n,0]_{i} \cong \bigotimes_{[i] \twoheadrightarrow [n]} UF(0) \cong
  \bigotimes_{[i] \twoheadrightarrow [n]} b(\F)[0] \cong b\left(\bigoplus_{[i]
    \twoheadrightarrow [n]} \F\right)[0] \cong b(K\F[n])[0]_{i}, \] and the
  simplicial structure maps are again the same. 
\end{proof}

Now we consider the consider the more complicated case, namely
$UK[n,q,k]$.  There is a cofibration sequence
\[ F(q)[n] \to C(q,k)[n] \to F(2^{k}q)[n+1] \] of chain complexes of
restricted vector spaces.  By the Dold-Kan correspondence this gives a
cofibration sequence
\[ K[n,q] \to K[n,q,k] \to K[n+1,2^{k}q] \]
of simplicial restricted vector spaces, and so a cofibration sequence
\[ U(K[n,q]) \to U(K[n,q,k]) \to U(K[n+1,2^{k}q])\] of simplicial
commutative $\F$-algebras by Proposition~\ref{propn:Ucolim}. We want
to apply the ``Serre spectral sequence'' of
Corollary~\ref{cor:SerreSpSeq} to this cofibration sequence to compute
$\pi_{*}U(K[n,q,k])$; in order to do this we first observe that the
map $U(K[n,q]) \to U(K[n,q,k])$ is a cofibration of simplicial graded
commutative algebras:
\begin{lemma}
  The map $U(K[n,q]) \to U(K[n,q,k])$ is almost-free in the sense of
  Definition~\ref{defn:almost}, and so is a cofibration of
  simplicial graded commutative algebras.
\end{lemma}
\begin{proof}
  We have
  \[ K[n,q,k]_{i} = \bigoplus_{[i] \twoheadrightarrow [n]} F(q) \oplus
  \bigoplus_{[i] \twoheadrightarrow [n+1]} F(2^{k}q)\] and
  $UK[n,q,k]_{i} \cong UK[n,q]_{i} \otimes S(V_{i})$ where $V_{i} =
  \bigoplus_{[i] \twoheadrightarrow [n+1]} \F[2^{k}q]$. It is clear
  that the simplicial structure maps $\psi^{*}$ take $UK[q,n]_{*}$ to
  itself, and are induced by maps between the $V_{n}$'s except when
  $\psi$ is such that for some $\beta$ and $\gamma$ we have $\beta
  \psi = d^{0}\gamma$, since this is the case when the differential
  in the chain complex occurs in the definition of $\psi^{*}$ for
  $K[n,q,k]$.

  But the map $\beta \colon [i] \to [n+1]$ is surjective and
  order-preserving, so it must send $0$ to $0$. Thus $\beta \psi$ will hit
  $0$ in $[n+1]$ for all $\beta$ if $\psi$ hits $0$ in $[i]$, in which
  case $\beta \psi$ cannot be of the form $d^{0}\gamma$. This is
  clearly the case for the degeneracies $s^{j} \colon [i+1] \to [i]$
  for all $j$ (as they are surjective) and the face maps $d^{j} \colon
  [i-1] \to [i]$ for $j \neq 0$. Thus the structure maps of
  $U(K[n,q,k])$ corresponding to all degeneracies and all face maps
  other than $d_{0}$ are induced from maps between the $V_{i}$.
\end{proof}

\begin{propn}\label{propn:UKnqkcollapse}
  The ``Serre spectral sequence'' for the cofibration sequence
  \[ U(K[n,q]) \xto{\alpha} U(K[n,q,k]) \xto{\beta} U(K[n+1,2^{k}q])\]
  collapses at the $E^{2}$-page.
\end{propn}
\begin{proof}
  By Corollary~\ref{cor:SerreSpSeq} the spectral sequence in question
  is a multiplicative spectral sequence of the form 
  \[ E^{2}_{s,t} = \pi_{s}(UK[n+1,2^{k}q]) \otimes \pi_{t-s}(UK[n,q])
  \Rightarrow \pi_{t}(UK[n,q,k]),\] with differentials $d^{r} \colon
  E^{r}_{s,t} \to E^{r}_{s-r,t-1}$. Write $\iota$ for the fundamental
  class in $\pi_{n}K[n,q] = E^{2}_{0,n}$ and $\kappa$ for the fundamental class in
  $\pi_{n+1}K[n+1,2^{k}q] = E^{2}_{n+1,n+1}$. Since the spectral
  sequence is multiplicative, it suffices to show that there are no
  non-zero differentials on the classes $\delta_{I}\iota$ and
  $\delta_{I}\kappa$. Clearly there are no possible non-zero
  differentials on $\delta_{I}\iota$, and the only differential on
  $\kappa$ that hits a non-zero group is $d^{n+1}$ --- but
  $d^{n+1}\kappa$ cannot be $\iota$ since they differ in internal
  grading. Moreover, the groups that might support differentials
  hitting $\kappa$ are all zero, so $\kappa$ must survive to $E^{\infty}$. 

  There is an obvious map of cofibration sequences from \[U(K[n,q]) \to
  U(K[n,q,k]) \to U(K[n+1,2^{k}q])\] to \[\F[0] \to U(K[n+1,2^{k}q])
  \xto{\id} U(K[n+1,2^{k}q]),\] and the map $\F[0] \to
  U(K[n+1,2^{k}q])$ is a cofibration since $U(K[n+1,2^{k}q])$ is
  free. Thus we get a morphism of spectral sequences, given on the
  $E^{2}$-page by projection to $\pi_{*}U(K[n+1,2^{k}q])$ and on the
  $E^{\infty}$-page by $\pi_{*}\beta$. This means that $\pi_{*}\beta$ must send
  $\kappa$ to the fundamental class $\kappa'$ in $\pi_{*}U(K[n+1,2^{k}q])$, and
  so for any admissible sequence $I$ the class $\delta_{I}\kappa$ is
  mapped to $\delta_{I}\kappa'$, which is
  non-zero. This implies that $\delta_{I}\kappa$ must also survive to
  $E^{\infty}$ for all $I$. By multiplicativity, this means the
  spectral sequence has no non-zero differentials, i.e. it collapses
  on the $E^{2}$-page.
\end{proof}

\begin{cor}\label{cor:piUKnqk}
  There is an isomorphism \[\pi_{*}U(K[n,q,k]) \cong
  \mathfrak{S}\F[n,q] \otimes \mathfrak{S}\F[n+1,2^{k}q]\] of algebras
  over the triple $\mathfrak{S}$.
\end{cor}
\begin{proof}
  We saw in the proof of Proposition~\ref{propn:UKnqkcollapse} that
  the map $\pi_{*}\beta \colon \pi_{*}UK[n,q,k] \to
  \pi_{*}UK[n+1,2^{k}q]$ is surjective. Since $\pi_{*}UK[n+1,2^{k}q]$
  is free, choosing a pre-image of the generator gives a map
  $\pi_{*}UK[n+1,2^{k}q] \to \pi_{*}UK[n,q,k]$ of
  $\mathfrak{S}$-algebras. Since the tensor product is the coproduct,
  we get a map
  \[\pi_{*}UK[n,q] \otimes \pi_{*}UK[n+1,2^{k}q] \to \pi_{*}UK[n,q,k]\] of
  $\mathfrak{S}$-algebras. Filter the left-hand side by degree and the
  right-hand side by the filtration from the ``Serre spectral
  sequence''. The collapse of this spectral sequence implies that this
  gives an isomorphism of the graded objects associated to the
  filtration, hence this map is an isomorphism of bigraded vector
  spaces and so also an isomorphism of $\mathfrak{S}$-algebras.
\end{proof}

Combining Lemmas~\ref{lem:UKnq} and \ref{lem:piUKnzero} and
Corollary~\ref{cor:piUKnqk} now completes the proof of
Theorem~\ref{thm:piU}. Applying this to the $E_{2}$-term of our
spectral sequence, we deduce the following:
\begin{cor}\label{cor:E2page}\ 
  \begin{enumerate}[(i)]
  \item If $E$ is a connected spectrum of finite type, the
    $E_{2}$-term of the spectral sequence for
    $\HH^{*}\Omega^{\infty}E$ is of the form
  \[ E_{2} \cong UD(\HH^{*}E)[0] \otimes
  \mathfrak{S}(\Sigma\Omega\mathbb{L}_{*>0}D(\HH^{*}E)) \otimes
  \mathfrak{S}(\Sigma\Omega_{1}\mathbb{L}_{*>0}D(\HH^{*}E)[1]).\]
\item If in addition the top squares in $\mathbb{L}_{*}D(\HH^{*}E))$
  are all zero for $* > 0$, then the $E_{2}$-term is given by
  \[ E_{2} \cong UD(\HH^{*}E)[0] \otimes
  \mathfrak{E}(\mathbb{L}_{*>0}D(\HH^{*}E)).\]
  \end{enumerate}
\end{cor}

\subsection{A Functorial Description of the Derived Functors}\label{subsec:functderU}
The description of $\pi_{*}UM$ for a simplicial restricted vector
space we obtained above is compatible with the products and
$\delta$-operations. However, in the case we're interested in $M$ is
the underlying simplicial restricted vector space of a simplicial
unstable $A$-module --- this means that there are also Steenrod
operations on $\pi_{*}UM$. We will now give a more functorial
description of $\pi_{*}UM$ that is also compatible with these
operations.

If $M$ is a simplicial unstable $A$-module, we have a natural
transformation $M \to \Sigma \Omega M$, which induces a map $\pi_{*}UM
\to \pi_{*}U\Sigma\Omega M$. Since the top squares in $\Sigma \Omega
M$ are all zero, as a simplicial commutative algebra $U \Sigma \Omega
M$ is isomorphic to $E \Sigma \Omega M$, and hence
$\pi_{*}U\Sigma\Omega M$ is isomorphic to $\mathfrak{E}(\pi_{*}\Sigma
\Omega M)$. We can also easily describe the action of the Steenrod
operations here, using the following observation of Dwyer:
\begin{propn}[Dwyer, {\cite[Proposition 2.7]{DwyerDivSqSpSeq}}]
\label{propn:SqDelta}
  Let $R$ be a simplicial unstable algebra over the Steenrod algebra;
  then $\pi_{*}R$ supports both higher divided squares and Steenrod
  operations. These are related as follows:
  \[ \Sq^{k}\delta_{i} =
  \begin{cases}
    0,& \text{$k$ odd,} \\
    \delta_{i} \Sq^{k/2}, & \text{$k$ even,}
  \end{cases}
  \]
  for $i \geq 2$. Moreover, if all squares in $R$ are zero, then the
  same is true for $i = 1$.
\end{propn}
\begin{proof}
  Write $\Sq := \Sq^{0} + \Sq^{1} + \cdots$. By the Cartan formula,
  $\Sq \colon R \to R$ is an algebra homomorphism. Since the
  operation $\delta_{i}$ is natural, this means $\delta_{i}\Sq =
  \Sq \delta_{i}$ in $\pi_{*}R$. Considering the homogeneous parts in
  each internal degree on both sides gives the result.
\end{proof}

We first consider the case where $M$ is a simplicial unstable
$A$-module such that $\pi_{0}M = 0$: 
\begin{propn}\label{propn:piUno0}
  Suppose $M$ is a levelwise projective simplicial unstable $A$-module such
  that $\pi_{0}M = 0$. Then there is a natural isomorphism of
  commutative bigraded $\F$-algebras
  \[ \pi_{*}UM \to
  \mathfrak{S}(\pi_{*}\Sigma\Omega M),\] 
  compatible with Steenrod operations and higher divided squares.
\end{propn}
\begin{proof}
  Observe that if $N$ is a simplicial unstable $A$-module such that
  $\pi_{0}N = 0$ and the top squares in $N$ vanish (such as $N =
  \Sigma \Omega M$), then we have an isomorphism
  $\mathfrak{E}(\pi_{*}N) \cong \mathfrak{S}(\pi_{*}N) \otimes
  \mathfrak{S}(\delta_{1}\pi_{*}N)$, compatible with the Steenrod
  operations. Moreover, the inclusion $\mathfrak{S}(\pi_{*}N)
  \hookrightarrow \mathfrak{E}(\pi_{*}N)$ and retraction
  $\mathfrak{E}(\pi_{*}N) \to \mathfrak{S}(\pi_{*}N)$ are compatible
  with Steenrod operations. Taking $N = \Sigma \Omega M$ we thus have
  a natural map of graded algebras
  \[ \pi_{*}UM \to \mathfrak{E}(\pi_{*}\Sigma\Omega M) \to
  \mathfrak{S}(\pi_{*}\Sigma\Omega M),\] compatible with Steenrod
  squares and $\delta$-operations. We will show that this map is an
  isomorphism.

  By Corollary~\ref{cor:RestrDecomp}, $M$ is weakly equivalent to a
  coproduct of $K[n,q]$'s and $K[n,q,k]$'s. We know that $U$ preserves
  weak equivalences and colimits by Proposition~\ref{propn:Ucolim},
  and $\Omega$ preserves coproducts and weak equivalences between
  levelwise projective objects. Thus it suffices to prove the result
  when $M$ is $K[n,q]$ and $K[n,q,k]$.

  In the first case, $\Sigma \Omega K[n,q]$ is $K(\F[n,q])$, and by
  Lemma~\ref{lem:UKnq} we know that $UK[n,q]$ is $SK(\F[n,q])$ for $n
  > 0$, so the
  map $U K[n,q] \to U \Sigma\Omega K[n,q]$ is the natural map $S
  K(\F[n,q]) \to E K(\F[n,q])$. On homotopy this is just the
  inclusion of the factor $\mathfrak{S}(\F[n,q])$.

  When $n = 0$, we have $\Sigma \Omega K[n,0] \simeq 0$, so
  \[\pi_{*}UK[n,0] \to \pi_{*}U\Sigma\Omega K[n,0] \to
  \mathfrak{S}(\pi_{*}\Sigma\Omega K[n,0])\] is the identity map
  on $\F[0]$ by Lemma~\ref{lem:piUKnzero}.

  For $UK[n,q,k]$ we consider the extension sequence $UK[n,q] \to
  UK[n,q,k] \to UK[n+1, 2^{k}q]$. On homotopy, this leads to a
  commutative diagram
 \[ \ltikzcd{%
   \pi_{*}UK[n,q] \arrow{r} \arrow{d} \pgfmatrixnextcell
   \mathfrak{S}(\pi_{*}\Sigma\Omega K[n,q]) \arrow{d}
   \\
      \pi_{*}UK[n,q,k] \arrow{r} \arrow{d} \pgfmatrixnextcell
   \mathfrak{S}(\pi_{*}\Sigma\Omega K[n,q,k]) \arrow{d}
   \\
      \pi_{*}UK[n+1,2^{k}q] \arrow{r} \pgfmatrixnextcell
   \mathfrak{S}(\pi_{*}\Sigma\Omega K[n+1,2^{k}q]).
}
  \]
  Here we have already shown that the top and bottom horizontal
  morphism are isomorphisms. But the chain complex $\Sigma \Omega
  C(q,k)$ is clearly $\F[q] \oplus \Sigma \F[2^{k}q]$, so the right
  vertical maps are a split extension sequence. Moreover, from the
  proof of Corollary~\ref{cor:piUKnqk} we know that the lower left
  vertical map is surjective, and that choosing a preimage of the
  generator gives an isomorphism $\mathfrak{S}(\pi_{*}\Sigma\Omega
  K[n,q]) \otimes \mathfrak{S}(\pi_{*}\Sigma\Omega K[n+1,2^{k}q])
  \isoto \pi_{*}UK[n,q,k]$. The composite of this with the map
  $\pi_{*}UK[n,q,k] \to \mathfrak{S}(\pi_{*}\Sigma\Omega K[n,q,k])$ is
  also an isomorphism (since it is determined by where it sends the
  generators). Thus by the 2-out-of-3 property the middle horizontal
  map here must also be an isomorphism, which completes the proof.
\end{proof}

For a general simplicial unstable $A$-module $M$ we have a projection
$M \to \pi_{0}M[0]$. Writing $M_{>0}$ for the fibre of this map, we
have a pushout square \nolabelcsquare{M_{>0}}{M}{0}{M_{0}} where
$M_{0}$ is weakly equivalent to $\pi_{0}M[0]$ (as can be seen from the
long exact sequence in homotopy groups). Thus we have an pushout
diagram \nolabelcsquare{UM_{>0}}{UM}{\F}{UM_{0}} of simplicial
unstable $A$-algebras. On homotopy we thus have maps $\pi_{*}UM_{>0}
\to \pi_{*}UM \to U\pi_{0}M[0]$, where the second map is an
isomorphism on $\pi_{0}$. We thus have a canonical map $U \pi_{0}M[0]
\to \pi_{*}UM$, and since the tensor product is the coproduct here we
get a map $\pi_{*}UM_{>0}\otimes U\pi_{0}M[0] \to \pi_{*}UM$,
compatible with all the operations in play. Moreover, this is an
isomorphism --- as usual, this follows from considering the case where
$M$ is $K[n,q]$ or $K[n,q,k]$.

\begin{thm}\label{thm:piUftrial}
  Suppose $M$ is a simplicial unstable $A$-module, and let $M' \to M$
  be a weak equivalence where $M'$ is levelwise projective. Then we
  have a natural  isomorphism of commutative bigraded $\F$-algebras
  \[ \mathfrak{S}(\pi_{*>0}\Sigma\Omega M') \otimes U \pi_{0}M[0] \isoto
  \pi_{*}UM,\]
  compatible with Steenrod operations and higher divided
  squares. Moreover, there is a short exact sequence
  \[ 0 \to \Sigma\Omega_{1}\pi_{*>0}M[1] \to \pi_{*>0}\Sigma\Omega M' \to
  \Sigma \Omega \pi_{*>0}M \to 0\]
  of graded unstable $A$-modules.
\end{thm}
\begin{proof}
  Since $U$ preserves weak equivalences by Proposition~\ref{propn:Ucolim}, we have $\pi_{*}UM' \cong
  \pi_{*}UM$, so the desired isomorphism follows from
  Proposition~\ref{propn:piUno0}. The short exact sequence is a
  consequence of the hyperhomology spectral sequence
  \[ \Sigma\Omega_{s}\pi_{t}M \Rightarrow \pi_{s+t}\Sigma\Omega M\]
  (see for example \cite{GrothendieckTohoku}), which collapses.
\end{proof}

\begin{cor}
  Let $X$ be a connected spectrum of finite type. In the infinite
  loops spectral sequence for $X$ the $E_{2}$-term is isomorphic to 
  \[ \mathfrak{S}(\mathbb{L}_{*>0}(\Sigma\Omega D)(\HH^{*}X)) \otimes
  UD(\HH^{*}X),\] compatibly with products, Steenrod operations, and
  higher divided squares. Moreover, there is a short exact sequence
  \[ 0 \to \Sigma\Omega_{1}\mathbb{L}_{* > 0}D(\HH^{*}X)[1] \to
  \mathbb{L}_{*>0}\Sigma\Omega D(\HH^{*}X) \to
  \Sigma \Omega \mathbb{L}_{*>0}D(\HH^{*}X) \to 0\]
  of graded unstable $A$-modules.
\end{cor}
\begin{proof}
  Recall that the $E_{2}$-term is obtained by applying $UD$ to a free simplicial
  resolution of $\HH^{*}X$ as an $A$-module. Applying $D$ to this free
  resolution gives a levelwise free simplicial unstable $A$-module, so
  this is immediate from Theorem~\ref{thm:piUftrial}.
\end{proof}

\section{Examples}\label{sec:Ex}
Corollary~\ref{cor:E2page} reduces the analysis of the $E_{2}$-term of
our spectral sequence to the computation of the derived functors of
$D$. In this section we will apply results about these functors from
the literature to describe the spectral sequence in two simple
examples.

\subsection{Eilenberg-Mac Lane Spectra}
The spectral sequence is clearly trivial for Eilenberg-Mac Lane spectra
of the form $\Sigma^{k}\HH\F$. As a slightly less trivial example we
can consider the Eilenberg-Mac Lane spectra $\Sigma^{k}\HH\ZZ$ and
$\Sigma^{k}\HH\ZZ /2^{n}$, where we must have $k > 0$ for our
convergence result to apply. The mod-2 cohomology of the spectrum
$\HH\ZZ$, originally computed by Serre~\cite{SerreCohlgyEM}, is the
$A$-module $A/\Sq^{1}$. Since $\Sq^{1}\Sq^{1} = 0$, this has a simple
free resolution, namely
\[ \cdots \to \Sigma^{2} A \xto{\cdot \Sq^{1}} \Sigma A \xto{\cdot
  \Sq^{1}} A.\] From this we see that, writing $F(n) = D(\Sigma^{n}A)$
for the free unstable $A$-module on a generator in degree $n$, the
derived functors $\mathbb{L}_{*}D(\HH^{*}\Sigma^{k}\HH\ZZ)$ are given
by the cohomology of the complex
\[ \cdots \to F(k+2) \xto{\cdot \Sq^{1}} F(k+1) \xto{\cdot \Sq^{1}}
F(k).\]
But it is easy to see that this complex is exact for $k > 0$, and so
$\mathbb{L}_{*}D(\HH^{*}\Sigma^{k}\HH\ZZ)$ is 0 for $* > 0$. It follows
that our spectral sequence has only a single column, and so collapses
to give \[\HH^{*}(K(\ZZ,k)) = \HH^{*}(\Omega^{\infty}\Sigma^{k}\HH\ZZ)
\cong U(F(k)/\Sq^{1}).\]

Similarly, the spectrum $\Sigma^{k}\HH\ZZ/2^{n}$ has cohomology
$A/\Sq^{1} \oplus \Sigma A/\Sq^{1}$, so again $D$ has no derived
functors and we get $\HH^{*}(K(\ZZ/2^{n}, k)) \cong U(F(k)/\Sq^{1})
\otimes U(F(k+1)/\Sq^{1})$. Of course, these results agree with
Serre's computations in \cite{SerreCohlgyEM}.

\subsection{Suspension Spectra}
In this subsection we consider the spectral sequence for infinite
loop spaces of the form $\Omega^{\infty}\Sigma^{\infty}X$, where $X$
is a connected space --- we will show, by a dimension-counting
argument, that the spectral sequence collapses in this case. 

The cohomology $\HH^{*}(\Sigma^{\infty}X) \cong \tilde{\HH}^{*}X$ is
an unstable $A$-module, and the derived functors
$\mathbb{L}_{*}(D)(M)$ for an unstable $A$-module $M$ were computed by
Lannes and Zarati~\cite{LannesZaratiDestab}. An alternative
computation (in the dual, homological, case), using a chain complex
originally due to Singer~\cite{SingerLoops2}, has been given by Kuhn
and McCarty~\cite{KuhnMcCarty}, and we will use their formulation of
the result. Before stating this, we must introduce some notation:
\begin{defn}
  Let $M$ be an $A$-module. Let $\mathcal{R}_{s}M$ be the quotient of
  the graded $\F$-vector space generated by symbols $Q^{I}x$ in degree
  $|x| + i_{1} + \cdots + i_{s}$, where $x \in M$ and $I =
  (i_{1},\ldots,i_{s})$, by the instability and Ad\'{e}m relations for
  the Dyer-Lashof algebra as well as linearity relations ($Q^{I}(x+y)
  = Q^{I}x + Q^{I}y$). This becomes an $A$-module via the (dual)
  Nishida relation
  \[ \Sq^{i}Q^{j}x = \sum_{k}\binom{j-k}{i-2k}
  Q^{i+j-k}\Sq^{k}x.\]

  Let $d \colon \mathcal{R}_{s}(\Sigma M) \to \mathcal{R}_{s-1}(M)$ be
  the map defined by $d(Q^{I}\sigma x) =
  Q^{i_{1},\ldots,i_{s-1}}(\Sq^{i_{s}+1}x)$; this is a map of
  $A$-modules. Writing $R_{s}M := \Sigma \mathcal{R}_{s}\Sigma^{s-1}$
  we can think of $d$ as a map $R_{s}M \to R_{s-1}M$.
\end{defn}
The result, in the form given by \cite[Theorems 4.22 and
4.34]{KuhnMcCarty}, is then:
\begin{thm}[Singer, Lannes-Zarati, Kuhn-McCarty]
  Let $M$ be an $A$-module.
  \begin{enumerate}[(i)]
  \item The sequence
    \[ \cdots \to R_{s}M \xto{d} R_{s-1}M \to \cdots \to R_{0}M \] is
    a chain complex, and $\HH_{*}(R_{*}M)$ is naturally isomorphic to
    $\mathbb{L}_{*}(D)(M)$.
  \item If $M$ is an unstable $A$-module, then the differential in
    $R_{*}M$ is zero and thus $\mathbb{L}_{*}D(M) \cong R_{*}M$.
  \item If $M$ is an unstable $A$-module, then $\mathbb{L}_{s}D(M)$ is
    an $s$-fold suspension of an unstable module.
  \end{enumerate}
\end{thm}

By (iii), it follows that for $M$ unstable all the top squares in
$\mathbb{L}_{s}D(M)$ are zero for $s > 0$. By Corollary~\ref{cor:E2page}
we therefore have an isomorphism
\[ \mathbb{L}_{*}(UD)(M) \cong U(M)[0] \otimes
\mathfrak{E}(\mathbb{L}_{*>0}D(M)).\]

Thus if $E$ is a spectrum such that $\HH^{*}E$ is an unstable
$A$-module, in the $E^{2}$-term of our spectral sequence for
$\HH^{*}(\Omega^{\infty}E)$ an element $v \in \HH^{k}E$ gives:
\begin{itemize}
\item $\sigma Q^{I} \sigma^{s-1} v$ in degree $(-s, k + |I| + s)$,
  where $I = (i_{1},\ldots,i_{s})$ is an allowable sequence,
  i.e. $i_{t} \leq 2 i_{t+1}$, and $i_{1} > i_{2} + \cdots + i_{s} +
  |v| + s-1$. For brevity we'll denote this element by $\bar{Q}^{I}v$.
\item For $J$ an admissible sequence of length $l$, 
  $\delta_{J}\bar{Q}^{I}v$ in degree 
  $(-s - |J|, 2^{l}(k + |I| + s))$.
\end{itemize}
The $E_{2}$-page of the spectral sequence is an exterior algebra on
these generators.

Now suppose $X$ is a connected space of finite type; then $\HH^{*}X \cong (\HH_{*}(X))^{\vee}$. In
this case the spectral sequence for $\Sigma^{\infty}X$ converges by
Theorem~\ref{thm:Conv}. We wish to compare the $E_{2}$-page to the
known cohomology $\HH^{*}(Q X) \cong (\HH_{*}(Q X))^{\vee}$. Recall
that the homology $\HH_{*}(QX)$ can be described in terms of the
Dyer-Lashof operations $Q^{j}$:
\begin{thm}[May, \cite{CohenLadaMay}]
  If $X$ is a space, the homology $\HH_{*}(QX)$ is a polynomial
  algebra on generators $Q^{J}v$ where $v$ ranges over a basis of
  $\HH_{*}(X)$, and $J = (j_{1}, \ldots, j_{s})$ is an \emph{allowable}
  sequence, meaning $j_{t} \leq 2 j_{t+1}$ for all $t$, and $j_{1} >
  j_{2} + \ldots + j_{s} + |v|$. The element $Q^{J}v$ is in degree
  $|v| + |J|$.
\end{thm}
To see that the spectral sequence must collapse, it suffices to prove
the following:

\begin{propn}
  There is a grading-preserving bijection between the exterior algebra
  generators $\bar{Q}^{I}x$, $\delta_{J}\bar{Q}^{I}x$ of the
  $E_{2}$-page and the Dyer-Lashof operations $Q^{K}x$, together with
  their powers $(Q^{K}x)^{2^{r}}$, for each $x$ in a basis for the
  reduced cohomology of $X$.
\end{propn}
\begin{proof}
  Write $k := |x|$. Observe that for any allowable sequence $I$ with
  $i_{1} > i_{2} + \cdots + i_{s} + k + s - 1$, the total degree of
  $\bar{Q}^{I}x$ is the same as the degree of $Q^{I}x$. However, there
  are more non-zero Dyer-Lashof operations on $x$ than those given by
  these sequences: we are missing those where $i_{2} + \cdots + i_{s}
  + k < i_{1} \leq i_{2} + \cdots + i_{s} + |x| + s - 1$. To relate
  these to the $E_{2}$-term, we change the indexing of the Dyer-Lashof
  operations: If $J = (j_{1}, \ldots, j_{s})$ is an allowable
sequence, then for $Q^{J}v$ to be non-zero (and not a square) in
$\HH_{*}(QX)$ we must have, for positive integers $l_{1}, \ldots,
l_{s}$:
\[ j_{i} = k + j_{s} + j_{s-1} + \cdots + j_{i+1} + l_{i}.\] The
allowability condition, expressed in terms of the $l_{i}$'s, says that
$l_{i} \leq l_{i+1}$. Thus there exist non-negative integers $a_{1},
\ldots, a_{s}$ (with $a_{1} > 0$) such that $l_{i+1} = l_{i} +
a_{i+1}$ (and $l_{1} = a_{1}$). In terms of the $a_{i}$'s the element
$Q^{J}x$ has degree
\[ 2^{s}k + \sum_{j = 1}^{s} \sum_{r = 1}^{j} 2^{j-1} a_{r} = 2^{s}k +
\sum_{r = 1}^{s} \left(\sum_{j = r}^{s} 2^{j-1}\right) a_{r} = 2^{s}k
+ \sum_{r=1}^{s} (2^{s} - 2^{r-1})a_{r}.\] Let's write $q^{a_{1},
  \ldots, a_{s}}x$ for the element $Q^{J}x$ with $J$ of this form. We
also extend the notation by writing $q^{0,a_{1},\ldots,a_{s}}x$ for
$(q^{a_{1},\ldots,a_{s}}x)^{2}$, etc.

Defining $\bar{q}^{a_{1},\ldots,a_{s}}x$ similarly, we see that
$\bar{q}^{a_{1},\ldots, a_{s}}x$ has the same degree as
$q^{a_{1}+(s-1),a_{2},\ldots,a_{s}}x$ in $\HH_{*}(QX)$.

Now suppose $\delta_{I}$ is an admissible sequence of
$\delta$-operations of length $l$. Then there exist non-negative
integers $r_{t}$ such that $i_{l} = r_{l} \geq 1$ and $i_{t} =
2i_{t+1} + r_{t}$ for $t < l$; in terms of the $r_{t}$'s the
admissibility criterion says that $r_{1} + \cdots + r_{l} \leq
s$, and $|I| = \sum_{i}(2^{i}-1)r_{i}$. Then the total degree of $\delta_{I}\bar{q}^{a_{1},\ldots,
  a_{s}}x$ is the same as the degree of $q^{K}x$ where \[K = (s-\sum
r_{t}, r_{1},\ldots, r_{2}, \ldots, r_{l-1}, a_{1}+r_{l}-1,
a_{2},\ldots, a_{s}).\]

To see that this gives a bijection between the generators, we describe
its inverse: For $q^{b_{1},\ldots, b_{\sigma}}x$ in $H_{*}(QX)$, let $L$ be the unique integer
with $0 \leq L < \sigma$ such that 
\[ b_{1} + \cdots + b_{L} + L < \sigma \leq b_{1} + \cdots + b_{L+1} + L +
1.\]
Then we define 
\[s := \sigma-L,\]
\[ r_{L} := s - b_{1} - \cdots - b_{L} - L, \]
\[ r_{t} := b_{t+1}, \qquad t = 1,\ldots, L-1,\]
\[ a_{1} := b_{L+1} - r_{L} + 1\]
\[ a_{i} := b_{L+i}, \qquad i = 2, \ldots, s\]
Then 
\[ (b_{1}, \ldots, b_{\sigma}) = (s - L - \sum_{t=1}^{l} r_{t}, r_{1},
\ldots, r_{L-1}, a_{1} +r_{L} - 1, a_{2}, \ldots, a_{s}),\] so
$q^{b_{1},\ldots,b_{\sigma}}$ corresponds to 
$\delta_{I}\bar{q}^{a_{1},\ldots,a_{s}}x$ where $I = (i_{1}, \ldots,
i_{L})$ is the admissible sequence determined by the $r_{t}$'s,
i.e. with $i_{t} := \sum_{j = t}^{L} 2^{j-t} r_{j}$.
\end{proof}

\begin{cor}
  For $X$ a connected space of finite type, the spectral sequence
  \[\mathbb{L}_{*}(UD)(\HH^{*}X) \Rightarrow \HH^{*}(QX)\] collapses at
  the $E_{2}$-page.
\end{cor}

\bibliographystyle{gtart}

\begin{thebibliography}{}
\providecommand\bibmarginpar{\leavevmode\marginpar}
\def\urlstyle#1{{\tt #1}}

\bibitem{BousfieldOpsDerFtr}
\textbf{A\,K Bousfield}, \emph{Operations on derived functors of non-additive
  functors} (1967)\ Unpublished.

\bibitem{BousfieldLocSpa}
\textbf{A\,K Bousfield}, \href{http://dx.doi.org/10.1016/0040-9383(79)90018-1}
  {\emph{The localization of spectra with respect to homology}}, Topology 18
  (1979) 257--281

\bibitem{BousfieldHlgySpSeq}
\textbf{A\,K Bousfield}, \href{http://dx.doi.org/10.2307/2374579} {\emph{On the
  homology spectral sequence of a cosimplicial space}}, Amer. J. Math. 109
  (1987) 361--394

\bibitem{BousfieldKanCompl}
\textbf{A\,K Bousfield}, \textbf{D\,M Kan}, \emph{Homotopy limits, completions
  and localizations}, Lecture Notes in Mathematics, Vol. 304, Springer-Verlag,
  Berlin-New York (1972)

\bibitem{CartanPuissDiv}
\textbf{H Cartan}, \emph{Puissances divisées}, Séminaire Henri Cartan 7
  ({1954--1955}) exp. no. 7\
  \url{http://www.numdam.org/item?id=SHC_1954-1955__7_1_A7_0}

\bibitem{CohenLadaMay}
\textbf{F\,R Cohen}, \textbf{T\,J Lada}, \textbf{J\,P May}, \emph{The homology
  of iterated loop spaces}, Lecture Notes in Mathematics, Vol. 533,
  Springer-Verlag, Berlin-New York (1976)

\bibitem{DoldSymmProd}
\textbf{A Dold}, \emph{Homology of symmetric products and other functors of
  complexes}, Ann. of Math. (2) 68 (1958) 54--80

\bibitem{DoldPuppeNichtFtr}
\textbf{A Dold}, \textbf{D Puppe}, \emph{Homologie nicht-additiver {F}unktoren.
  {A}nwendungen}, Ann. Inst. Fourier Grenoble 11 (1961) 201--312

\bibitem{DwyerDivSqSpSeq}
\textbf{W\,G Dwyer}, \href{http://dx.doi.org/10.2307/1998013} {\emph{Higher
  divided squares in second-quadrant spectral sequences}}, Trans. Amer. Math.
  Soc. 260 (1980) 437--447

\bibitem{DwyerHtpyOps}
\textbf{W\,G Dwyer}, \href{http://dx.doi.org/10.2307/1998012} {\emph{Homotopy
  operations for simplicial commutative algebras}}, Trans. Amer. Math. Soc. 260
  (1980) 421--435

\bibitem{GoerssUnstProj}
\textbf{P\,G Goerss}, \href{http://dx.doi.org/10.1112/plms/s3-53.3.539}
  {\emph{Unstable projectives and stable {${\rm Ext}$}: with applications}},
  Proc. London Math. Soc. (3) 53 (1986) 539--561

\bibitem{GoerssAQ}
\textbf{P\,G Goerss}, \emph{On the {A}ndr\'e-{Q}uillen cohomology of
  commutative {${\bf F}_2$}-algebras}, Ast\'erisque  (1990) 169

\bibitem{GoerssLada}
\textbf{P\,G Goerss}, \textbf{T\,J Lada},
  \href{http://dx.doi.org/10.2307/2160555} {\emph{Relations among homotopy
  operations for simplicial commutative algebras}}, Proc. Amer. Math. Soc. 123
  (1995) 2637--2641

\bibitem{GrothendieckTohoku}
\textbf{A Grothendieck}, \emph{Sur quelques points d'alg\`ebre homologique},
  T\^ohoku Math. J. (2) 9 (1957) 119--221

\bibitem{HackneyCosInfLoop}
\textbf{P Hackney}, \href{http://dx.doi.org/10.1016/j.jpaa.2012.10.002}
  {\emph{Operations in the homology spectral sequence of a cosimplicial
  infinite loop space}}, J. Pure Appl. Algebra 217 (2013) 1350--1377

\bibitem{Isaacson}
\textbf{S\,B Isaacson}, \href{http://dx.doi.org/10.1016/j.jpaa.2010.08.001}
  {\emph{Symmetric cubical sets}}, J. Pure Appl. Algebra 215 (2011) 1146--1173
  \xox{arXiv}{0910.4948}

\bibitem{KanFtrsCSS}
\textbf{D\,M Kan}, \emph{Functors involving c.s.s. complexes}, Trans. Amer.
  Math. Soc. 87 (1958) 330--346

\bibitem{KuhnMcCarty}
\textbf{N Kuhn}, \textbf{J McCarty},
  \href{http://dx.doi.org/10.2140/agt.2013.13.687} {\emph{The mod 2 homology of
  infinite loopspaces}}, Algebr. Geom. Topol. 13 (2013) 687--745
  \xox{arXiv}{1109.3694}

\bibitem{LannesZaratiDestab}
\textbf{J Lannes}, \textbf{S Zarati},
  \href{http://dx.doi.org/10.1007/BF01168004} {\emph{Sur les foncteurs
  d\'eriv\'es de la d\'estabilisation}}, Math. Z. 194 (1987) 25--59

\bibitem{MacLaneWorking}
\textbf{S Mac~Lane}, \emph{Categories for the working mathematician}, volume~5
  of \emph{Graduate Texts in Mathematics}, second edition, Springer-Verlag, New
  York (1998)

\bibitem{MillerDeloop}
\textbf{H Miller}, \href{http://projecteuclid.org/euclid.pjm/1102805992}
  {\emph{A spectral sequence for the homology of an infinite delooping}},
  Pacific J. Math. 79 (1978) 139--155

\bibitem{MillerSullivan}
\textbf{H Miller}, \href{http://dx.doi.org/10.2307/2007071} {\emph{The
  {S}ullivan conjecture on maps from classifying spaces}}, Ann. of Math. (2)
  120 (1984) 39--87

\bibitem{MillerSullivanCorr}
\textbf{H Miller}, \href{http://dx.doi.org/10.2307/1971212} {\emph{Correction
  to: ``{T}he {S}ullivan conjecture on maps from classifying spaces''}}, Ann.
  of Math. (2) 121 (1985) 605--609

\bibitem{Powell}
\textbf{G\,M\,L Powell}, \href{http://dx.doi.org/10.1007/s40306-014-0062-3}
  {\emph{On the derived functors of destabilization at odd primes}}, Acta Math.
  Vietnam. 39 (2014) 205--236 \xox{arXiv}{1101.0226}

\bibitem{QuillenHtpclAlg}
\textbf{D\,G Quillen}, \emph{Homotopical algebra}, Lecture Notes in
  Mathematics, No. 43, Springer-Verlag, Berlin-New York (1967)

\bibitem{RezkSimplAlg}
\textbf{C Rezk}, \href{http://dx.doi.org/10.1016/S0166-8641(01)00057-8}
  {\emph{Every homotopy theory of simplicial algebras admits a proper model}},
  Topology Appl. 119 (2002) 65--94

\bibitem{SchwedeShipleyAlgMod}
\textbf{S Schwede}, \textbf{B\,E Shipley},
  \href{http://dx.doi.org/10.1112/S002461150001220X} {\emph{Algebras and
  modules in monoidal model categories}}, Proc. London Math. Soc. (3) 80 (2000)
  491--511

\bibitem{SchwedeShipleyMon}
\textbf{S Schwede}, \textbf{B Shipley},
  \href{http://dx.doi.org/10.2140/agt.2003.3.287} {\emph{Equivalences of
  monoidal model categories}}, Algebr. Geom. Topol. 3 (2003) 287--334

\bibitem{SerreCohlgyEM}
\textbf{J-P Serre}, \emph{Cohomologie modulo {$2$} des complexes
  d'{E}ilenberg-{M}ac{L}ane}, Comment. Math. Helv. 27 (1953) 198--232

\bibitem{SingerLoops2}
\textbf{W\,M Singer}, \href{http://dx.doi.org/10.1016/0022-4049(80)90044-4}
  {\emph{Iterated loop functors and the homology of the {S}teenrod algebra.
  {II}. {A} chain complex for {$\Omega ^{k}_{s}M$}}}, J. Pure Appl. Algebra 16
  (1980) 85--97

\bibitem{SingerNewChCx}
\textbf{W\,M Singer}, \href{http://dx.doi.org/10.1017/S0305004100058746}
  {\emph{A new chain complex for the homology of the {S}teenrod algebra}},
  Math. Proc. Cambridge Philos. Soc. 90 (1981) 279--292

\bibitem{Turner}
\textbf{J\,M Turner}, \href{http://dx.doi.org/10.1007/s002220000096} {\emph{On
  simplicial commutative algebras with vanishing {A}ndr\'e-{Q}uillen
  homology}}, Invent. Math. 142 (2000) 547--558

\end{thebibliography}

\end{document}